\documentclass{article}
\usepackage{geometry,amsmath,amssymb,enumerate,cite,comment,graphicx,caption,xcolor,booktabs,url}
\usepackage[title]{appendix}
\usepackage{hyperref}
\usepackage[thmmarks,amsmath,thref]{ntheorem}
\newtheorem{theorem}{Theorem}[section]

\newtheorem{lemma}{Lemma}[section]
\newtheorem{assumption}{Assumption}[section]
\newtheorem{corollary}{Corollary}[section]

\theorembodyfont{\normalfont}
\newtheorem*{remark}{Remark}
\theoremsymbol{{$\square$}}
\newtheorem*{proof}{Proof}
\newtheorem*{proofof}{Proof\hspace{-1pt}}
\newcommand{\red}{}
\newcommand{\blue}{}
\renewcommand{\d}{\mathrm{d}}

\newcommand{\ep}{\varepsilon}
\newcommand{\vp}{\varphi}
\newcommand{\W}{\mathcal{W}}
\newcommand{\Ge}{\geqslant}
\newcommand{\Le}{\leqslant}
\renewcommand{\P}{\mathcal{P}}
\newcommand{\D}{\mathcal{D}}
\newcommand{\C}{\mathcal{C}}

\newcommand{\N}{\mathcal{N}}
\newcommand{\M}{\mathcal{M}}
\newcommand{\RB}{\mathrm{RB}}
\newcommand{\Law}{\mathrm{Law}}

\newcommand{\setzero}{\setlength{\itemsep}{0pt}}
\newcommand{\di}{\displaystyle}
\numberwithin{equation}{section}
\title{Error Analysis of Time-Discrete Random Batch Method for Interacting Particle Systems and Associated Mean-Field Limits}
\newcommand{\BICMR}{Beijing International Center for Mathematical Research, Peking University, Beijing, 100871, P. R. China.}
\author{Xuda Ye\thanks{\BICMR
Email: abneryepku@pku.edu.cn} 
\and
Zhennan Zhou\thanks{\BICMR
Email: zhennan@bicmr.pku.edu.cn}}
\begin{document}
\maketitle
\begin{abstract}
The random batch method provides an efficient algorithm for computing statistical properties of a canonical ensemble of interacting particles. In this work, we study the error estimates of the fully discrete random batch method, especially in terms of approximating the invariant distribution.
The triangle inequality framework proposed in this paper is a convenient approach to estimate the long-time sampling error of the numerical methods. Using the triangle inequality framework, we show that the long-time error of the discrete random batch method is $O(\sqrt{\tau} + e^{-\lambda t})$, where $\tau$ is the time step and $\lambda$ is the convergence rate which does not depend on the time step $\tau$ or the number of particles $N$. Our results also apply to the McKean--Vlasov process, which is the mean-field limit of the interacting particle system as the number of particles $N\rightarrow\infty$.
\end{abstract}
{\small
}

\section{Introduction}
Simulation of large interacting particle systems (IPS) has always been an appealing research topic in computational physics \cite{cpc_1,cpc_2} and computational chemistry \cite{cpc_3,cpc_4}. It is not only because the IPS itself is an important model in molecular dynamics and quantum mechanics, but also because the IPS has a mathematically well-defined mean-field limit \cite{mean_1,mean_2,mean_3,mean_4} as the number of particles tends to infinity. The mean-field dynamics of the IPS is a distribution-dependent SDE, also known as the McKean-Vlasov process (MVP), has been frequently used in statistical physics to describe the ensemble behavior of a system of particles \cite{MVP_1,MVP_2}.
In this paper we focus on a simple IPS model, which is evolved by the overdamped Langevin dynamics with only pairwise interactions.

Consider the system of $N$ particles in $\mathbb R^{Nd}$ represented by a collection of position variables $X_t = \{X_t^i\}_{i=1}^N$, where the subscript $t\Ge0$ denotes the evolution time and each particle $X_t^i\in\mathbb R^d$ is evolved by the overdamped Langevin dynamics
\begin{equation}
	\d X_t^i = \bigg(b(X_t^i) + \frac1{N-1}
	\sum_{j\neq i}K(X_t^i-X_t^j)\bigg)\d t + \sigma\d W_t^i.
	\label{IPS}
\end{equation}
Here, $b:\mathbb R^d\rightarrow\mathbb R^d$ is the drift force, $K:\mathbb R^d\rightarrow\mathbb R^d$ is the interaction force, $\sigma>0$ is the diffusion coefficient, and $\{W_t^i\}_{i=1}^N$ are $N$ independent Wiener processes in $\mathbb R^d$. Formally, the mean-field limit of \eqref{IPS} as $N\rightarrow\infty$ is the MVP represented by a single position variable $\bar X_t\in\mathbb R^d$
\begin{equation}
\begin{aligned}
	\d \bar X_t & = \bigg(b(\bar X_t) + \int_{\mathbb R^d}
	K(\bar X_t-z)\nu_t(\d z)\bigg)
	\d t + \sigma \d W_t, \\
	\nu_t & = \Law(\bar X_t).
\end{aligned}
\label{MVP}
\end{equation}
Here, $\Law(\cdot)$ denotes the distribution law of a random variable, and $W_t$ is the Wiener process in $\mathbb R^d$.
The convergence mechanism of the IPS \eqref{IPS} towards the MVP \eqref{MVP} as $N\rightarrow\infty$ has been systemically studied in the theory of the propagation of chaos \cite{chaos_4,chaos_3}.

The goal of this paper is to study the sampling accuracy of the numerical methods for the IPS \eqref{IPS} and the MVP \eqref{MVP}. To characterize the sampling accuracy of a numerical method at different time scales, it is reasonable to ask the following two questions:
\begin{enumerate}
\setzero
\item In the finite-time level, does the method produce accurate trajectories?
\item In the long-time level, does the method sample the correct invariant distribution?
\end{enumerate}
To be specific, suppose the numerical method for the IPS \eqref{IPS} with the time step $\tau$ produces a discrete-time trajectory $\{\tilde X_n\}_{n\Ge0}$ in $\mathbb R^{Nd}$, where the subscript $n$ is a nonnegative integer representing the number of iterations.
Then we expect $\{\tilde X_n\}_{n\Ge0}$ is a good approximation the discrete-time IPS trajectory $\{X_{n\tau}\}_{n\Ge0}$; and for sufficiently large $n$,
the numerical distribution law $\Law(\tilde X_n)$ is close to the invariant distribution of the IPS \eqref{IPS}.

In this paper we shall consider the Euler--Maruyama scheme, which provides a simple numerical method for the IPS \eqref{IPS}. Fix the time step $\tau$ and define $t_n:=n\tau$, then the IPS \eqref{IPS} is approximated by a system of particles $\tilde X_n = \{\tilde X_n^i\}_{i=1}^N$, where each particle $\tilde X_n^i\in\mathbb R^d$ in the time interval $[t_n,t_{n+1})$ is updated by the following stochastic equation
\begin{equation}
	\tilde X_{n+1}^i = \tilde X_n^i +
	\bigg(b(\tilde X_n^i) + 
	\frac1{N-1}\sum_{j\neq i}K(\tilde X_n^i-\tilde X_n^j)\bigg)
	\tau + \sigma(W_{t_{n+1}}^i - W_{t_n}^i),
	\label{d-IPS}
\end{equation}
which we shall refer to as the discrete IPS thereafter. We note that the discrete IPS \eqref{d-IPS} is also known as the stochastic particle method \cite{numerical_1,numerical_2}, which can be applied in a wide class of MVPs, and the associated error analysis can be found in \cite{numerical_3,numerical_4,numerical_5,numerical_6}.
To update the discrete IPS \eqref{d-IPS} in a single time step, we need to compute all the pairwise interactions $K(\tilde X_t^i - \tilde X_t^j)$, hence the computational cost per time step is $O(N^2)$. Such huge complexity brings great burden when $N$ is large.

The Random Batch Method (RBM) proposed in \cite{rbm_1} resolves the complexity burden in the discrete IPS \eqref{d-IPS} with a simple idea: for each time step, compute the interaction forces within small random batches.
For each $n\Ge0$, let the index set $\{1,\cdots,N\}$ be randomly divided into $q$ batches $\D=\{\C_1,\cdots,\C_q\}$, where each batch $\C\in\D$ has the equal size $p = N/q$ where the integer $p\Ge2$. 
We compute the interaction force between two particles only when their indices $i,j$ belong to the same batch. The discrete IPS \eqref{d-IPS} is then approximated by the discrete random batch interacting particle system (discrete RB--IPS), represented by a system of particles $\tilde Y_n = \{\tilde Y_n^i\}_{i=1}^N$ in $\mathbb R^{Nd}$, where each particle $\tilde Y_n^i\in\mathbb R^d$ is updated by
\begin{equation}
	\tilde Y_{n+1}^i = \tilde Y_n^i +
	\bigg(b(\tilde Y_n^i) + 
	\frac1{p-1}\sum_{j\neq i,j\in\C}K(\tilde Y_n^i-\tilde Y_n^j)\bigg)
	\tau + \sigma(W_{t_{n+1}}^i-W_{t_n}^i),
	~~i\in \C.
	\label{d-RB-IPS}
\end{equation}
Here, $\C\in\D$ is the unique batch containing $i$.
For the next time interval, the previous division $\D$ is discarded and another random division is employed.
The discrete RB--IPS \eqref{d-RB-IPS} requires only $O(Np)$ rather than $O(N^2)$ complexity to compute the interaction forces in a time step, which is a significant advance in simulation efficiency.

Nowadays the RBM has become a prominent simulation tool for large particle systems.
It is not only a highly efficient numerical method for complex chemical systems \cite{rbm_2,rbm_3,rbm_4,rbm_5}, but also accelerates the particle ensemble methods \cite{pem_1,pem_2,pem_3} for optimization or solving PDEs. 
There have been some theoretical results on the error analysis of the RBM, but they mainly focus on the continuous-time random batch interacting particle system (RB--IPS, defined in \eqref{RB-IPS}). In the finite-time level, it was proved in \cite{rbme_1} that the strong and weak error are $O(\sqrt{\tau})$ and $O(\tau)$ respectively; while in the long-time level, the authors of \cite{rbme_2} applied the reflection coupling \cite{reflect_1,reflect_2} to show that the RB--IPS has uniform geometric ergodicity, and the Wasserstein-1 distance between the invariant distributions of the IPS and the RB--IPS is bounded by $O(\sqrt{\tau})$. However, the error analysis of the discrete RB--IPS \eqref{d-RB-IPS} is not a direct consequence of the results for the RB--IPS. Moreover, the long-time behavior of the discrete RB--IPS \eqref{d-RB-IPS} poses additional challenges because it is fairly non-trivial to obtain an explicit convergence rate towards the invariant distribution.
Therefore, it is necessary to perform the error analysis for the discrete RB--IPS \eqref{d-RB-IPS}, which is the main task of this paper.

The triangle inequality framework proposed in this paper is our main technique to study the long-time sampling error. This framework is inspired from Mattingly \cite{tri_2,tri_3} and Durmus \cite{tri_4}, and can be  conveniently applied in a wide class of numerical methods. For a given stochastic process and the corresponding numerical method, the triangle inequality framework is able to utilize the ergodicity of the original process and the finite-time error analysis to estimate the long-time error.
Furthermore, with the triangle inequality framework, it is easy to produce an explicit convergence rate, which is independent of the time step $\tau$ or other parameters.
In particular, for the IPS \eqref{IPS} and its corresponding numerical method---the discrete RB--IPS \eqref{d-RB-IPS}, the convergence rate is independent of the number of particles $N$.

Before we elaborate the principle of the triangle inequality framework in Section 2, we state the main results of this paper. These results are proved by combining the triangle inequality framework and the error analysis results for the RB--IPS in \cite{rbme_1,rbme_2}.
\begin{enumerate}
\item \textbf{(Theorem \ref{theorem:strong error}) The finite-time strong error is $O(\sqrt{\tau})$.} 

When the IPS \eqref{IPS} and the discrete RB--IPS \eqref{d-RB-IPS} are driven by the same initial value and Wiener processes, there exists a positive constant $C=C(T)$ such that
\begin{equation}
\sup_{0\Le n\Le T/\tau}
\frac1N \sum_{i=1}^N
\mathbb E|X_{n\tau}^i - \tilde Y_n^i|^2\Le C\tau.
\label{IPS finite error}
\end{equation}
The constant $C$ does not on $N,\tau$ or $p$.
\item \textbf{(Theorem \ref{theorem:IPS sample quality})
The long-time sampling error is $O(\sqrt{\tau}+e^{-\lambda t})$.}

When the interaction force $K$ is moderately large, there exist constants $C,\lambda>0$ such that 
\begin{equation}
	\W_1(\pi,\Law(\tilde Y_n)) \Le C\sqrt{\tau} + Ce^{-\lambda n\tau},~~~~
	\forall n\Ge0,
\label{IPS long error}
\end{equation}
where $\pi\in\P(\mathbb R^{Nd})$ is the invariant distribution of the IPS \eqref{IPS} and $\W_1$ is the normalized Wasserstein-1 distance defined in \eqref{Wasserstein}. The constants $C,\lambda$ do not depend on $N,\tau$ or $p$.
\end{enumerate}
In the long-time sampling error \eqref{IPS long error},
the order of accuracy in the time step $\tau$ may not be optimal. This is because we have used the strong error estimate \eqref{IPS finite error} in the triangle inequality framework to prove \eqref{IPS long error} (see  Section 2.3). Nevertheless, the convergence rate $\lambda$ does not depend on the number of particles $N$, the time step $\tau$ or the batch size $p$.

Using the results in the propagation of chaos \cite{chaos_1,chaos_3},
we show that the discrete RB--IPS \eqref{d-RB-IPS} is also a reliable numerical method for the MVP \eqref{MVP}. In particular, the invariant distribution of the MVP \eqref{MVP} can be approximated by the empirical distribution of the discrete RB--IPS \eqref{d-RB-IPS} by choosing the number of particles $N$ sufficiently large.
\begin{enumerate}
\setzero
\item \textbf{(Corollary \ref{corollary:RB-IPS strong})  The finite-time strong error is $O\big(\sqrt{\tau} + \frac1{\sqrt{N}}\big)$.} 

When $N$ duplicates of the MVP \eqref{MVP} and the discrete RB--IPS \eqref{d-RB-IPS} are driven by the same initial value and Wiener processes, there exists a positive constant $C=C(T)$ such that
\begin{equation}
\sup_{0\Le n\Le T/\tau}
\frac1N \sum_{i=1}^N
\mathbb E|\bar X_{n\tau}^i - \tilde Y_n^i|^2\Le C\tau + \frac CN,
\label{MVP finite error}
\end{equation}
where $\{\bar X_t^i\}_{t\Ge0}$ is the $i$-th duplicate of the MVP \eqref{MVP}. The constant $C$ does not depend on $N,\tau$ or $p$.
\item \textbf{(Corollary \ref{corollary:MVP sample quality}) The long-time sampling error is $O\big(\sqrt{\tau}+e^{-\lambda t}+\frac1{\sqrt{N}}\big)$.}

When the interaction force $K$ is moderately large, there exist constants $C,\lambda>0$ such that 
\begin{equation}
	\mathbb E\big[\W_1(\bar\pi,\tilde\mu_{n\tau}^{\RB})\big] \Le C\sqrt{\tau} + Ce^{-\lambda n\tau} + \frac{C}{\sqrt{N}},~~~~
	\forall n\Ge0,
\label{MVP long error}
\end{equation}
where $\bar\pi\in\P(\mathbb R^d)$ is the invariant distribution of the MVP (\ref{MVP}), and $\tilde \mu_{n\tau}^{\RB}$ is the empirical measure of the $N$-particle system $\{\tilde Y_n^i\}_{i=1}^N$, i.e.,
\begin{equation}
	\tilde \mu_{n\tau}^{\RB}(x) = \frac1N\sum_{i=1}^N 
	\delta(x-\tilde Y_n^i)\in\P(\mathbb R^d).
\end{equation}
The constants $C,\lambda$ do not depend on $N,\tau$ or $p$.
\end{enumerate}
The paper is organized as follows. In Section 2 we introduce the triangle inequality framework for estimating the long-time sampling error. In Section 3 we prove \eqref{IPS finite error}\eqref{IPS long error} for the IPS \eqref{IPS}. In Section 4 we prove \eqref{MVP finite error}\eqref{MVP long error} for the MVP \eqref{MVP}.
\section{Triangle inequality for long-time error analysis}
In general, the long-time error analysis of a numerical method is much more difficult than the finite-time error analysis, whose proof is standard and can be found in textbooks, e.g., Chapter 7.5 of \cite{asp}. Nevertheless, Mattingly \cite{tri_2,tri_3} and Durmus \cite{tri_4} proposed a special strategy---which we refer to as the triangle inequality framework in this paper---to address the problem of the long-time error analysis. The idea of this framework is simple. In addition to the known results in the finite-time error analysis, one only needs the geometric ergodicity of the stochastic dynamics to perform the long-time error analysis. In short words, the geometric ergodicity with the finite-time error yields the long-time error.

In the rest part of this section, we first review the original approaches employed in \cite{tri_2,tri_3,tri_4} for the long-time error analysis. Motivated by their results, we propose a general lemma on the long-time error analysis. Finally, we demonstrate why the triangle inequality framework can be applied in a wide class of stochastic dynamics, including the discrete RB--IPS \eqref{d-RB-IPS}.

\subsection{A historical review}
The geometric ergodicity is the key property to describe the long-time behavior of a stochastic process, and is essential to build up the triangle inequality framework. For simplicity, consider the continuous-time stochastic process $\{X_t\}_{t\Ge0}$, whose transition probability is $(p_t)_{t\Ge0}$.
Let $\P(\mathbb R^d)$ be the space of all probability distributions on $\mathbb R^d$, then for any $\nu\in\P(\mathbb R^d)$, $\nu p_t\in\P(\mathbb R^d)$ is the distribution law of $X_t$ provided $X_0\sim\nu$.
Given the metric $d(\cdot,\cdot)$ on $\P(\mathbb R^d)$,
the stochastic process $\{X_t\}_{t\Ge0}$ is said to have \emph{geometric ergodicity}, if it has an invariant distribution $\pi\in\P(\mathbb R^d)$, and
there exist positive constants $C,\beta$ such that
\begin{equation}
	d(\nu p_t,\pi) \Le Ce^{-\beta t}d(\nu,\pi),~~~~
	\forall \nu\in\P(\mathbb R^d).
	\label{geometric ergodicity}
\end{equation}
In other words, the distribution law $\nu p_t$ converges to the invariant distribution $\pi$ exponentially, and $\beta$ is the convergence rate.

Now consider another stochastic process $\{\tilde X_t\}_{t\Ge0}$ with transition probability $(\tilde p_t)_{t\Ge0}$, which can be viewed as an approximation to the original process $\{X_t\}_{t\Ge0}$.
For example, $\{X_t\}_{t\Ge0}$ is the solution to an SDE, while $\{\tilde X_t\}_{t\Ge0}$ is given by the Euler--Maruyama scheme. To characterize the long-time error of $\{\tilde X_t\}_{t\Ge0}$, the following two questions are proposed in \cite{tri_3}:
\begin{enumerate}
\setzero
\item Does $\{\tilde X_t\}_{t\Ge0}$ has a unique invariant distribution $\tilde\pi\in\P(\mathbb R^d)$?
\item If so, what is the difference between $\pi$ and $\tilde\pi$?
\end{enumerate}
The first question can be directly addressed by the Harris ergodic theorem \cite{tri_2,var_1,reflect_3,invariant}. For the second question, a special triangle inequality was adopted in \cite{tri_3} to estimate the difference between  $\pi$ and $\tilde\pi$. Under the same metric $d(\cdot,\cdot)$, assume the \emph{finite-time difference relation} between the distribution laws $\nu p_t,\nu\tilde p_t$ is known, that is, for any $T>0$  there exists a constant $\ep(T)$ such that
\begin{equation}
	\sup_{0\Le t\Le T}
	d(\nu p_t,\nu \tilde p_t) \Le \ep(T),~~~~
	\forall \nu \in \P(\mathbb R^d).
	\label{finite-time diff}
\end{equation}
Here, $T$ is a reference evolution time of the processes $\{X_t\}_{t\Ge0},\{\tilde X_t\}_{t\Ge0}$, and $(p_t)_{t\Ge0},(\tilde p_t)_{t\Ge0}$ are the corresponding transition probabilities.
If we choose the the metric $d(\cdot,\cdot)$ to be the Wasserstein-1 distance, and derive the finite-time difference relation \eqref{finite-time diff}  from the standard strong error estimate, then the error bound $\ep(T)$ is approximately
\begin{equation}
	\ep(T) \approx O(e^{CT}\sqrt{\tau}),
	\label{ep bound}
\end{equation}
where $\tau>0$ is the time step used in time discretization.
\eqref{ep bound} implies that $\ep(T)$ grows exponentially with the evolution time $T$, and $\ep(T)$ is bounded by $O(\sqrt{\tau})$ with a fixed $T$.

Provided the geometric ergodicity \eqref{geometric ergodicity} and the finite-time difference relation \eqref{finite-time diff}, we can now use the triangle inequality to estimate $d(\pi,\tilde\pi)$. In fact, for any $T>0$, we have
\begin{align}
	d(\pi,\tilde\pi) & = d(\pi p_T,\tilde\pi \tilde p_T) \notag \\
	& \Le d(\pi p_T,\tilde \pi p_T) + d(\tilde \pi p_T, \tilde\pi\tilde p_T) \notag \\
	& \Le Ce^{-\beta T} d(\pi,\tilde\pi) + \ep(T).
	\label{triangle inequality}
\end{align}
Hence if we choose $T = T_0$ in \eqref{triangle inequality}
to satisfy $Ce^{-\beta T_0} = 1/2$, then
\begin{equation}
	d(\pi,\tilde\pi) \Le 2\ep(T_0),
\end{equation}
which measures the difference between the invariant distributions $\pi$ and $\tilde\pi$. Since $T_0$ is a fixed value, we have $\ep(T_0) \approx O(\sqrt\tau)$, hence approximately $d(\pi,\tilde\pi) \Le O(\sqrt{\tau})$.

The triangle inequality used in \eqref{triangle inequality} is essentially the same with Remark 6.3 of \cite{tri_3}, and also previously appeared in \cite{tri_1,tri_2}. The benefit of the triangle inequality \eqref{triangle inequality} is obvious: it does not require the ergodicity of the approximation $\{\tilde X_t\}_{t\Ge0}$ to estimate the difference between $\pi$ and $\tilde\pi$. It only requires the geometric ergodcity of the original process $\{X_t\}_{t\Ge0}$, and the finite-time difference relation \eqref{finite-time diff}.
The drawback of the triangle inequality \eqref{triangle inequality} is that it does not tell how fast the distribution law of $\{\tilde X_t\}_{t\Ge0}$ converges to the invariant distribution $\tilde\pi\in\P(\mathbb R^d)$. Although the Harris ergodic theorem ensures that $\nu \tilde p_t$ converges to $\tilde\pi$ exponentially \cite{tri_2,tri_3}, it is usually difficult to make the convergence rate independent of the time step $\tau$ (see Theorem 7.3 of \cite{tri_2} for example).

In a recent paper \cite{tri_4}, the authors have utilized the geometric ergodicity and the triangle inequality to estimate the long-time sampling error of a given numerical method. Instead of calculating the difference between invariant distributions $d(\pi,\tilde\pi)$ directly, one turns to estimate $d(\nu\tilde p_t,\pi)$ for large $t$, that is, the difference between the numerical distribution law $\nu\tilde p_t$ and the true invariant distribution $\pi$.
For large $t$, $d(\nu\tilde p_t,\pi)$ can be interpreted as the long-time sampling error of the approximation $\{\tilde X_t\}_{t\Ge0}$. Also, one avoids computing the numerical invariant distribution $\tilde\pi$ directly. Although the proof strategies used in \cite{tri_3,tri_4} are quite different, it is clear that the triangle inequality plays an important role in estimating the long-time sampling error.

Based on the original triangle inequality adopted in \cite{tri_3}, and the idea of using $d(\nu\tilde p_t,\pi)$ instead of $d(\pi,\tilde\pi)$ in \cite{tri_4}, we propose the triangle inequality framework in the next subsection.
By choosing the metric $d(\cdot,\cdot)$ to be the Wasserstein-1 distance, we expect the long-time sampling error $d(\nu\tilde p_t,\pi)$ is bounded by
\begin{equation}
	d(\nu \tilde p_t,\pi) \Le O(\sqrt{\tau} + e^{-\lambda t}), ~~~~\forall t>0,
	\label{long-time sampling error}
\end{equation}
where the constant $\lambda>0$ does not depend on the time step $\tau$. Clearly, $d(\nu\tilde p_t,\pi)$ consists of two parts: the finite-time strong error $O(\sqrt{\tau})$ and the exponential convergence part $O(e^{-\lambda t})$. Although $\lambda$ does not indicate the geometric ergodicity of the approximation $\{\tilde X_t\}_{t\Ge0}$ itself, it does reveal the fact that the sampling efficiency of $\{\tilde X_t\}_{t\Ge0}$ can be uniform in the time step $\tau$.

We summarize the major differences between our work and the results in \cite{tri_2,tri_3,tri_4}. First, our work considers the numerical methods for the IPS \eqref{IPS}, which is a multi-particle system rather than a single particle. The geometric ergodicity of the IPS \eqref{IPS} is guaranteed by the reflection coupling \cite{reflect_1,reflect_2}, while their results mainly rely on the Harris ergodic theorem. This also leads to a difference in the choice of the metric $d(\cdot,\cdot)$: we shall always employ the normalized Wasserstein-1 distance, while their results mainly involve the weighted total variation \cite{var_1}. Second, the numerical method in our work involves the random batch approximations, which is more complicated than the standard Euler-Maruyama scheme. Finally, the triangle inequality used in this work is a variant of \eqref{triangle inequality} in \cite{tri_3} rather than the one used in \cite{tri_4}.
\subsection{Main lemma for the long-time error estimate}
We state the main lemma for the long-time error estimate, \red{which is the key conclusion of the triangle inequality framework.}
\begin{lemma}
\label{lemma:general}
Let $\{X_t\}_{t\Ge0}$, $\{\tilde X_t\}_{t\Ge0}$ be stochastic processes in $\mathbb R^d$ with transition probabilities $(p_t)_{t\Ge0}$, $(\tilde p_t)_{t\Ge0}$.
\blue{Given the metric $d(\cdot,\cdot)$ on $\P(\mathbb R^d)$,}
assume \red{$(p_t)_{t\Ge0}$ has an invariant distribution $\pi\in\P(\mathbb R^d)$ and} there exist constants $C,\beta>0$ such that
\begin{equation*}
	d(\nu p_t,\pi) \Le Ce^{-\beta t}d(\nu,\pi),~~~~
	\forall \nu\in\P(\mathbb R^d).
\end{equation*}
and for any $T>0$, there exists a constant $\ep(T)$ such that
\begin{equation*}
	\sup_{0\Le t\Le T}
	d(\nu \tilde p_t,\nu p_t) \Le \ep(T),~~~~
	\forall \nu\in\P(\mathbb R^d).
\end{equation*}
Then there exist constants $T_0,\lambda>0$ such that
\begin{equation}
	d(\nu \tilde p_t,\pi) \Le 2\ep(T_0) + 2M_0 e^{-\lambda t},~~~~
	\forall t\Ge0,
	\label{prop2 long time error}
\end{equation}
where $\di M_0 := \sup_{s\in [0,T_0]} d(\nu\tilde p_s,\pi)$.
\end{lemma}
\begin{proof}
We still estimate $d(\nu\tilde p_t,\pi)$ using the triangle inequality. For any $T>0$ and $t\Ge T$,
\begin{align*}
	d(\nu\tilde p_t,\pi) & \Le d(\nu\tilde p_{t-T}\tilde p_T,\nu\tilde p_{t-T} p_T) + 
	d(\nu \tilde p_{t-T} p_T,\pi p_T) \\
	& \Le \ep(T) + Ce^{-\beta T}
	d(\nu\tilde p_{t-T},\pi).
\end{align*}
By choosing $T = T_0$ such that $Ce^{-\beta T_0}=1/2$, we have
\begin{equation}
	d(\nu \tilde p_t,\pi) \Le \ep(T_0) + \frac12 d(\nu \tilde p_{t-T_0},\pi),~~~~
	\forall t\Ge T_0.
	\label{triangle ii}
\end{equation}
By induction on the integer $n\Ge0$, we obtain
\begin{equation}
	d(\nu \tilde p_t,\pi) \Le 2\bigg(1-\frac1{2^n}\bigg)\ep(T_0) + \frac1{2^n} d(\nu \tilde p_{t-nT_0},\pi),~~~~
	\forall t\Ge nT_0.
\end{equation}
For any $t\in[0,+\infty)$, there exists a unique integer $n\Ge0$ such that $t\in[nT_0,(n+1)T_0)$. Then
\begin{equation}
	d(\nu\tilde p_t,\pi) \Le 2\ep(T_0) + 2^{1-t/{T_0}} \sup_{s\in[0,T_0]} d(\nu\tilde p_s,\pi),
\end{equation}
which implies the long-time error estimate \eqref{prop2 long time error} with $\lambda = \ln 2/T_0$.
\end{proof}

The conditions in Lemma \ref{lemma:general} are exactly the geometric ergodicity \eqref{geometric ergodicity} and the finite-time difference relation \eqref{finite-time diff}, and the result \eqref{prop2 long time error} characterizes the long-time sampling error of the stochastic process $\{\tilde X_t\}_{t\Ge0}$.
\red{The triangle inequality used in \eqref{triangle ii} is essential in the proof of Lemma \ref{lemma:general}, which is the reason that Lemma \ref{lemma:general} is referred to as the triangle inequality framework.}
In particular, when $d(\cdot,\cdot)$ is the Wasserstein-1 distance, $\ep(T_0)$ is of order $O(\sqrt{\tau})$, and thus we recover the result in \eqref{long-time sampling error}.
Now we briefly summarize the pros and cons of the triangle inequality framework.
\begin{enumerate}
\setzero
\item It requires only the geometric ergodicity of the original dynamics $\{X_t\}_{t\Ge0}$. The existence of the invariant distribution for $\{\tilde X_t\}_{t\Ge0}$ is not required. This allows us to study a wide class of numerical methods, including the IPS and the methods with stochastic gradient or random batch approximations.
\item It produces an explicit convergence rate $\lambda>0$, which can be easily made independent of the time step $\tau$ and other parameters. In fact, $\lambda$ is uniquely determined by the parameters $C,\beta$ in the geometric ergodicity condition \eqref{geometric ergodicity}. In the discrete IPS \eqref{d-IPS} and the discrete RB--IPS \eqref{d-RB-IPS}, $\lambda$ is independent of the number of particles $N$.
\item The geometric ergodicity condition \eqref{geometric ergodicity} is very restrictive in the choice of the metric $d(\cdot,\cdot)$. 
\red{Here, the metric $d(\cdot,\cdot)$ must be symmetric in its two arguments and satisfy the triangle inequality.}
As a consequence, the convergence in entropy
\begin{equation}
	\red{H(\nu p_t|\pi) \Le Ce^{-\beta t} H(\nu|\pi),~~~~
	\forall \nu\in\P(\mathbb R^d)}
\end{equation}
cannot be used to prove the geometric ergodicity condition \eqref{geometric ergodicity} because \red{the relative entropy $H(\cdot|\pi)$ is not symmetric, despite the fact that it is stronger than the Wasserstein distance (Talagrand's inequality \cite{Talagrand}) and the total variation (Pinsker's inequality).}
\item The finite-time difference relation \eqref{finite-time diff} must be derived with a metric at least stronger than $d$, which \blue{might} make the order of accuracy not optimal. For example, when $d$ is the Wasserstein-1 distance, \eqref{finite-time diff} can be naturally derived from the strong error estimate, but the order of accuracy is only $O(\sqrt{\tau})$. It is still challenging for the triangle inequality framework to yield better accuracy in the time step.
\end{enumerate}
In short words, as long as the original dynamics $\{X_t\}_{t\Ge0}$ satisfies the geometric ergodicity condition \eqref{geometric ergodicity} in a specific metric $d(\cdot,\cdot)$, and the finite-time error analysis is valid in the metric $d(\cdot,\cdot)$, then we can use the triangle inequality to estimate the long-time sampling error. For example, $\{X_t\}_{t\Ge0}$ can be the IPS \eqref{IPS}, the MVP \eqref{MVP} or the Hamiltonian Monte Carlo, where the geometric ergodicity of $\{X_t\}_{t\Ge0}$ is guaranteed by the reflection coupling \cite{reflect_2,reflect_3,reflect_HMC}.

Finally, we \blue{remark} that the triangle inequality framework is \blue{remotely} reminiscent of the well-known Lax equivalence theorem in numerical analysis. Here, the geometric ergodicity serves as the stability and it helps translate the finite-time error estimate to the long-time error estimate without sacrificing the accuracy order.
\subsection{Application in the interacting particle system}
A significant advantage of the triangle inequality framework is that it naturally applies to the IPS (not necessarily in the form of \eqref{IPS}). When sampling an IPS, we naturally expect the error bound to be independent of the number of particles $N$. This is in general a difficult problem in stochastic analysis, and even more in the case of the long-time sampling error. Nevertheless, the requirement of the uniform-in-$N$ error bound can be explicitly interpreted in the triangle inequality framework.

In order to make the long-time sampling error \eqref{prop2 long time error} independent of the number of particles $N$, we need to satisfy the following two conditions.
\begin{enumerate}
\setzero
\item The finite-time error bound $\ep(T_0)$ is independent of $N$ (for fixed $T_0$);
\item The exponential convergence rate $\beta$ is independent of $N$.
\end{enumerate}
The first condition is relatively easy to obtain because $\ep(T_0)$ only relates to the finite-time error analysis.
If the multi-particle system has a mean-field limit as $N\rightarrow\infty$, the theory of propagation of chaos usually provides a convenient tool to study $\ep(T_0)$, see \cite{mvpp_1,numerical_4,numerical_5} for example.

The second condition is more demanding because it requires the IPS to have uniform geometric ergodicity in a specific metric $d(\cdot,\cdot)$. The Harris ergodic theorem is not suitable to prove the uniform ergodicity because it is difficult to quantify the the minorization condition in high dimensions \cite{reflect_3}. The uniform log-Sobolev inequality proved in \cite{uniform_log_Sobolev} has a uniform-in-$N$ convergence rate, but the relative entropy used to quantify the convergence is not a metric. Therefore, the most natural choice for the metric $d(\cdot,\cdot)$ in the IPS is the Wasserstein distance, and the uniform geometric ergodicity can be verified by the reflection coupling \cite{reflect_1,reflect_2}.

When both conditions are satisfied, we can use the triangle inequality framework to estimate the long-time sampling error of a large variety of numerical methods, although the order of accuracy is not optimal. In particular, for the first time we show that the discrete RB--IPS \eqref{d-RB-IPS}, as a time-discretization of the Random Batch Method, possesses a long-time error bound independent of the number of particles $N$ and the time step $\tau$.
\section{Error analysis of discrete RB--IPS for IPS}
In this section we analyze the error of the discrete RB--IPS \eqref{d-RB-IPS}, as an approximation to the IPS \eqref{IPS}. In Section 3.1, we derive the strong error in the finite time. In Section 3.2, we prove the uniform-in-time moment estimates for the discrete RB--IPS \eqref{d-RB-IPS}, which is necessary for the long-time error estimate. In Section 3.3, we briefly review the geometric ergodicity of the IPS \eqref{IPS} derived by the reflection coupling.
In Section 3.4, we combine the results above with the triangle inequality framework to derive the long-time error in the normalized Wasserstein-1 distance.

For the convenience of analysis, we also introduce the continuous-time random batch interacting particle system (RB--IPS), which is represented by a system of particles $Y_t = \{Y_t^i\}_{i=1}^N$ in $\mathbb R^{Nd}$, where each particle $Y_t^i\in\mathbb R^{Nd}$ in the time interval $[t_n,t_{n+1})$ is evolved by the following SDE
\begin{equation}
	\d Y_t^i = \bigg(
	b(Y_t^i) + 
	\frac1{p-1}
	\sum_{j\neq i,j\in\C}
	K(Y_t^i - Y_t^j)\bigg)\d t + 
	\sigma\d W_t^i,~~
	i\in \C,~~t\in[t_n,t_{n+1}).
	\label{RB-IPS}
\end{equation}
Here, $\D = \{\C_1,\cdots,\C_q\}$ is the batch division used in the time interval $[t_n,t_{n+1})$; and for each $i\in\{1,\cdots,N\}$, $\C\in\D$ is the unique batch that contains $i$. The error analysis for the RB--IPS \eqref{RB-IPS} can be found in \cite{rbme_1,rbme_2}.

We also list in Table 1 the notations of all dynamics involved in this paper and their corresponding transition probabilities, invariant distributions and equation numbers.
\begin{center}
\begin{tabular}{c|cccc}
\toprule
dynamics & notation & transition prob. & invariant dist. & equation \\
\midrule
IPS & $X_t\in\mathbb R^{Nd}$ & $(p_t)_{t\Ge0}$ & $\pi\in\P(\mathbb R^{Nd})$ & \eqref{IPS} \\
MVP & $\bar X_t\in\mathbb R^d$ & $(\bar p_t)_{t\Ge0}$ & $\bar\pi\in\P(\mathbb R^d)$ & \eqref{MVP} \\
discrete IPS & $\tilde X_n\in\mathbb R^{Nd}$ & $(\tilde p_{n\tau})_{n\Ge0}$ & -- & \eqref{d-IPS} \\
RB--IPS & $Y_t\in\mathbb R^{Nd}$ & $(q_{n\tau})_{n\Ge0}$ & -- & \eqref{RB-IPS} \\
discrete RB--IPS & $\tilde Y_n\in\mathbb R^{Nd}$ & $(\tilde q_{n\tau})_{n\Ge0}$ & -- & \eqref{d-RB-IPS} \\
\bottomrule
\end{tabular}
\captionof{table}{Notations of IPS, MVP, discrete IPS, RB--IPS and discrete RB--IPS.}
\end{center}
Here, `--' in the invariant distribution column means that the existence of such distribution is not required in our analysis.
\subsection{Strong error in finite time}
The discrete RB--IPS \eqref{d-RB-IPS} deviates from the IPS \eqref{IPS} for two reasons: time discretization and random batch divisions at each time step. Therefore, it is natural to analyze the impact of these two factors separately. Among the four dynamics: IPS \eqref{IPS}, discrete IPS \eqref{d-IPS}, RB--IPS \eqref{RB-IPS} and discrete RB--IPS \eqref{d-RB-IPS}, we focus on the following two types of strong error estimates.
\begin{enumerate}
\setzero
\item \textbf{Time discretization.}
\begin{equation}
	\mbox{discrete RB--IPS vs RB--IPS:} \sup_{0\Le n\Le T/\tau}\bigg(
		\frac1N\sum_{i=1}^N
		\mathbb E|Y^i_{n\tau} - \tilde Y^i_n|^2\bigg), 
	\label{strong error td}
\end{equation}
\item \textbf{Random batch divisions.}
\begin{equation}
	\mbox{RB--IPS vs IPS:}  \sup_{0\Le n\Le T/\tau}\bigg(
			\frac1N\sum_{i=1}^N
			\mathbb E|X^i_{n\tau} - Y^i_{n\tau}|^2\bigg).
	\label{strong error rb}
\end{equation}
\end{enumerate}
Here, we assume the four dynamics \eqref{IPS}\eqref{d-IPS}\eqref{RB-IPS}\eqref{d-RB-IPS} are in the synchronous coupling, i.e., they are driven by the same Wiener processes $\{W_t^i\}_{i=1}^N$, the same random batch divisions (if required) at each time step, and the same initial value $X_0$, where $X_0$ is a random variable on $\mathbb R^{Nd}$ with $\Law(X_0) = \nu$. Note that the discrete RB--IPS \eqref{d-RB-IPS} deviates from the RB--IPS \eqref{RB-IPS} only due to time-discretization because we impose the same random batch divisions for these two dynamics.

Once we obtain the strong error estimates \eqref{strong error td}\eqref{strong error rb}, the strong error between the discrete RB--IPS \eqref{d-RB-IPS} and the IPS \eqref{IPS} defined by
\begin{equation}
	\mbox{discrete RB--IPS vs IPS:}
	\sup_{0\Le n\Le T/\tau}\bigg(
	\frac1N\sum_{i=1}^N
	\mathbb E|X^i_{n\tau} - \tilde Y^i_n|^2\bigg).
	\label{strong error}
\end{equation}
directly follows from the triangle inequality. In the following we estimate \eqref{strong error td}\eqref{strong error rb} respectively.
\subsubsection*{Strong error due to time discretization}
Before we begin to estimate \eqref{strong error td}, it is convenient to introduce the strong error below
\begin{equation}
	\mbox{discrete IPS vs IPS:}  \sup_{0\Le n\Le T/\tau}\bigg(
			\frac1N\sum_{i=1}^N
			\mathbb E|X^i_{n\tau} - \tilde X^i_n|^2\bigg).
		\label{strong error td0}
\end{equation}
Since the both \eqref{strong error td}\eqref{strong error td0} origin from time discretization, we may apply similar methods to estimate \eqref{strong error td}\eqref{strong error td0}.
As in the standard routine in the strong error analysis, we impose the global Lipschitz and boundedness condition on the drift force $b$ and the interaction force $K$ as follows.
\begin{assumption}
\label{assumption:bounded}
For the drift force $b:\mathbb R^d\rightarrow\mathbb R^d$, there exists a constant $L_0$ such that
\begin{equation}
	|b(x)| \Le L_0(|x|+1),~~~
	|\nabla b(x)|\Le L_0,~~~
	\forall x\in\mathbb R^d.
	\label{boundedness 0}
\end{equation}
For the interaction force $K:\mathbb R^d\rightarrow\mathbb R^d$, there exists a constant $L_1$ such that
\begin{equation}
	\max\{|K(x)|,|\nabla K(x)|,|\nabla^2 K(x)|\} \Le L_1,~~~~\forall x\in\mathbb R^d.
	\label{boundedness 1}
\end{equation}
\end{assumption}

In the IPS \eqref{IPS}, define the perturbation force of the $i$-th particle by
\begin{equation}
	\gamma^i(x):=\frac1{N-1}\sum_{j\neq i}
	K(x^i-x^j),~~~~
	\forall x\in\mathbb R^{Nd},
	\label{IPS gamma}
\end{equation}
and the total force applied to the $i$-th particle by $b^i(x) = b(x^i) + \gamma^i(x)$.
Then the IPS \eqref{IPS} and the discrete IPS \eqref{d-IPS} can be simply written as
\begin{equation}
	\d X_t^i = b^i(X_t)\d t + \sigma\d W_t^i,~~
	\tilde X_{n+1}^i = \tilde X_n^i + b^i(\tilde X_n)\tau + 
	\sigma (W_{t_{n+1}}^i-W_{t_n}^i),~~i=1,\cdots,N.
	\label{IPS XX}
\end{equation}
According to \eqref{boundedness 1},
it is easy to verify $\gamma^i(x)$ is uniformly bounded by $L_1$, and
\begin{equation}
	|\gamma^i(x) - \gamma^i(y)|\Le L_1|x^i-y^i| + 
	\frac{L_1}{N-1}\sum_{j\neq i} |x^j-y^j|.
\end{equation}
Summation over $i$ yields the global Lipschitz condition for the perturbation force
\begin{equation}
	\sum_{i=1}^N |\gamma^i(x) - \gamma^i(y)| \Le 2L_1
	\sum_{i=1}^N |x^i-y^i|,~~~~
	\forall x,y\in\mathbb R^{Nd}.
\end{equation}

In the RB--IPS \eqref{RB-IPS}, suppose
the index set $\{1,\cdots,N\}$ is divided to $\D = \{\C_1,\cdots,\C_q\}$ to form the random batch dynamics in the time interval $[t_n,t_{n+1})$. 
In this case we slightly abuse the notation and again define the perturbation force by
\begin{equation}
	\gamma^i(x) = \frac1{p-1}\sum_{j\neq i,j\in\C}
	K(x^i-x^j),~~~~
	\forall x\in\mathbb R^{Nd},
\end{equation}
then with the new total force $b^i(x) = b(x^i) + \gamma^i(x)$, the RB--IPS \eqref{RB-IPS} and the discrete RB--IPS \eqref{d-RB-IPS} are simply given by
\begin{equation}
	\d Y_t^i = b^i(Y_t)\d t + \sigma\d W_t^i,~~
	\tilde Y_{n+1}^i = \tilde Y_n^i + b^i(\tilde Y_n)\tau + 
	\sigma (W_{t_{n+1}}^i-W_{t_n}^i),~~i=1,\cdots,N.
	\label{RB-IPS YY}
\end{equation}
Although \eqref{RB-IPS YY} is very similar to \eqref{IPS XX}, we stress that \eqref{RB-IPS YY} is valid only in the time step $[t_n,t_{n+1})$ due to the random batch divisions, and the formulation of $\gamma^i(x)$ varies in different time steps. Nevertheless, $\gamma^i(x)$ is uniformly bounded by $L_1$ regardless of the batch division $\D$. Also, we have
\begin{equation}
	|\gamma^i(x) - \gamma^i(y)| \Le L_1|x^i-y^i| + 
	\frac{L_1}{p-1}\sum_{j\neq i,j\in\C}|x^j-y^j|.
\end{equation}
Summation over $i\in\C$ gives
\begin{equation}
	\sum_{i\in\C}|\gamma^i(x) - \gamma^i(y)|
	\Le 2L_1\sum_{i\in\C}|x^i-y^i|,
\end{equation}
and summation over $\C\in\D$ gives
\begin{equation}
	\sum_{i=1}^N |\gamma^i(x) - \gamma^i(y)| \Le 2L_1\sum_{i=1}^N |x^i-y^i|.
\end{equation}
Therefore, the global Lipschitz condition still holds true for the random batch dynamics.

Based on the observation of $\gamma^i(x)$ above, we can prove the following results.
\begin{lemma}
\label{lemma:IPS moment}
Under Assumption \ref{assumption:bounded},
if there exists a constant $M_2$ such that
$$
	\max_{1\Le i\Le N}\mathbb E|X_0^i|^2\Le M_2,
$$
then there exists a constant $C=C(L_0,L_1,M_2,T,\sigma)$ such that
\begin{equation}
	\sup_{0\Le t\Le T}
	\mathbb E|X_t^i|^2 \Le C,~
	\sup_{t\in[t_n,t_{n+1}\wedge T)}\mathbb E
	|X_t^i - X_{t_n}^i|^2 \Le C\tau,
	\label{IPS moment finite time}
\end{equation}
and
\begin{equation}
	\sup_{0\Le t\Le T}
	\mathbb E|Y_t^i|^2 \Le C,~
	\sup_{t\in[t_n,t_{n+1}\wedge T)}\mathbb E
	|Y_t^i - Y_{t_n}^i|^2 \Le C\tau.
\end{equation}
\end{lemma}
The proof of Lemma \ref{lemma:IPS moment} is in Appendix.
The proof only requires the fact that $|\gamma^i(x)|  \Le L_1$.
\begin{theorem}
\label{theorem:IPS strong}
Under Assumption \ref{assumption:bounded}, if there exists a constant $M_2$ such that
$$
	\max_{1\Le i\Le N}\mathbb E|X_0^i|^2\Le M_2,
$$
then there exists a constant $C=C(L_0,L_1,M_2,T,\sigma)$ such that
\begin{equation}
	\sup_{0\Le n\Le T/\tau}\bigg(
	\frac1N\sum_{i=1}^N
	\mathbb E|X^i_{n\tau} - \tilde X^i_n|^2\bigg) \Le C\tau
\end{equation}
and
\begin{equation}
\sup_{0\Le n\Le T/\tau}\bigg(
	\frac1N\sum_{i=1}^N
	\mathbb E|Y^i_{n\tau} - \tilde Y^i_n|^2\bigg) \Le C\tau.
\end{equation}
\end{theorem}
The proof of Theorem \ref{theorem:IPS strong} is in Appendix.
The proof uses the fact that $\gamma^i(x)$ is global Lipschitz.
\begin{remark}We have some remarks on Theorem \ref{theorem:IPS strong}.
\begin{enumerate}
\item
If one employs a constant time step $\tau$,
the global Lipschitz condition on the drift force $b$ is necessary to ensure the stability of the numerical method. Even for an ergodic SDE, the Euler--Maruyama scheme can be unstable when $b$ is not globally Lipschitz, see the example in Section 6.3 of \cite{tri_2}. If there is only local Lipschitz condition on $b$, the readers may refer to \cite{numerical_3,numerical_5} for the discussion of other types of Euler--Maruyama schemes.
\item The constant $C$ depends on the second moments of the initial distribution $\nu\in\P(\mathbb R^{Nd})$, which is characterized by the constant $M_2$ in Theorem \ref{theorem:IPS strong}.
\end{enumerate}
\end{remark}
\subsubsection*{Strong error due to random batch divisions}
We compare the trajectory difference between the IPS \eqref{IPS} and the RB--IPS \eqref{RB-IPS}, which are both exactly integrated in the time interval $[t_n,t_{n+1})$. Recall that the strong error in this case is 
$$
	\sup_{0\Le n\Le T/\tau}\bigg(
	\frac1N\sum_{i=1}^N
	\mathbb E|X^i_{n\tau} - Y^i_{n\tau}|^2\bigg),
$$
where \eqref{IPS}\eqref{RB-IPS} are driven by the same Wiener processes $\{W_t^i\}_{i=1}^N$ and the same initial random variable $X_0\sim\nu$. The estimate of the strong error above directly follows Theorem 3.1 in \cite{rbme_1}, and we restate their result here.
\begin{theorem}
\label{theorem:RBM strong error}
Under Assumption \ref{assumption:bounded}, if there exists a constant $M_4$ such that
$$
	\max_{1\Le i\Le N}\mathbb E|X_0^i|^4\Le M_4,
$$
then there exists a constant $C=C(L_0,L_1,M_4,T,\sigma)$ such that
\begin{equation}
	\sup_{0\Le t\Le T}\bigg(
	\frac1N\sum_{i=1}^N
	\mathbb E|X^i_t - Y_t^i|^2\bigg) \Le C\bigg(\frac{\tau}{p-1}+\tau^2\bigg).
\end{equation}
\end{theorem}
\begin{remark}
We have some remarks on Theorem \ref{theorem:RBM strong error}:
\begin{enumerate}
\item Compared to Theorem \ref{theorem:IPS strong}, Theorem \ref{theorem:RBM strong error} requires the initial distribution $\nu$ has finite fourth-order moments rather than second-order moments.
This is because in the original proof in \cite{rbme_1}, the authors used the second-order Taylor expansion to to estimate the $L^2$ norm of
$$
	K(Y_t^i - Y_t^j) - K(Y_{n\tau}^i - Y_{n\tau}^i),
$$
which naturally produces the fourth order moments.
\item If the linear growth condition of $b(x)$ in \eqref{boundedness 0} is replaced by $|b(x)| \Le L_0(|x|+1)^q$ for some $q\Ge2$, then the initial distribution $\nu$ should have finite $2q$-th order moments. In this paper we only consider the case of $q=2$.
\end{enumerate}
\end{remark}

Using Theorems \ref{theorem:IPS strong} and \ref{theorem:RBM strong error}, we can now estimate the strong error \eqref{strong error} between the discrete RB--IPS \eqref{d-RB-IPS} and the IPS \eqref{IPS}.
\begin{theorem}
\label{theorem:strong error}
Under Assumption \ref{assumption:bounded}, if there exists a constant $M_4$ such that
$$
	\max_{1\Le i\Le N}\mathbb E|X_0^i|^4\Le M_4,
$$
then there exists a constant $C=C(L_0,L_1,M_4,T,\sigma)$ such that
\begin{equation}
	\sup_{0\Le n\Le T/\tau}\bigg(
	\frac1N\sum_{i=1}^N
	\mathbb E|X^i_{n\tau} - \tilde Y_n^i|^2\bigg) \Le C\tau.
\end{equation}
\end{theorem}
Theorem \ref{theorem:strong error} \red{implies} that applying random batch divisions does not \red{worsen} the order of strong error. However, the constant $C$ in Theorem \ref{theorem:strong error} can be much larger than in Theorem \ref{theorem:IPS strong}.

A direct consequence of Theorem \ref{theorem:strong error} is the finite-time error estimate of the discrete RB--IPS \eqref{d-RB-IPS} in the Wasserstein-2 distance. For given distributions $\mu,\nu\in \P(\mathbb R^{Nd})$, the normalized Wasserstein-$p$ distance between $\mu,\nu$ is defined by
\begin{equation}
	\W_p^p(\mu,\nu) = 
	\inf_{\gamma\in\Pi(\mu,\nu)}
	\int_{\mathbb R^{Nd}\times\mathbb R^{Nd}}
	\bigg(
	\frac1N\sum_{i=1}^N |x^i-y^i|^p
	\bigg)
	\gamma(\d x\d y),
	\label{Wasserstein}
\end{equation}
where $\Pi(\mu,\nu)$ is the transport plans between $\mu$ and $\nu$. The readers may also refer to \cite{chaos_3} for a thorough introduction to the normalized Wasserstein distance.

Let $(p_t)_{t\Ge0},(\tilde p_{n\tau})_{n\Ge0},(q_{n\tau})_{n\Ge0},(\tilde q_{n\tau})_{n\Ge0}$ be the transition probabilities of the four dynamics \eqref{IPS}\eqref{d-IPS}\eqref{RB-IPS}\eqref{d-RB-IPS} respectively. Then for given initial distribution $\nu\in\P(\mathbb R^{Nd})$, the distribution laws of $X_{n\tau}^i,\tilde X_n^i,Y_{n\tau}^i,\tilde Y_n^i$ are $\nu p_{n\tau},\nu \tilde p_{n\tau},\nu q_{n\tau},\nu \tilde q_{n\tau}$ respectively (see the notations in Table 1).
Here we note that $(p_t)_{t\Ge0}$ defines a Markov process, while $(\tilde p_{n\tau})_{n\Ge0},(q_{n\tau})_{n\Ge0},(\tilde q_{n\tau})_{n\Ge0}$ only define discrete-time Markov chains because of the random batch divisions at each time step. Although formally the transition probability $(q_t)_{t\Ge0}$ for the RB--IPS \eqref{RB-IPS} can be defined for any $t\Ge0$, $(q_t)_{t\Ge0}$ does not form a Markov semigroup.

Now we have the $\W_2$ error estimate for the discrete IPS \eqref{d-IPS} and the discrete RB--IPS \eqref{d-RB-IPS}.
\begin{corollary}
\label{corollary:IPS finite error}
Under Assumption \ref{assumption:bounded}, if there exists a constant $M_4$ such that
$$
	\max_{1\Le i\Le N}\int_{\mathbb R^{Nd}}
	|x^i|^4\nu(\d x)\Le M_4,
$$
then there exists a constant $C=C(L_0,L_1,M_4,T,\sigma)$ such that
\begin{equation}
\max\bigg\{
\sup_{0\Le n\Le T/\tau}
\W_2(\nu p_{n\tau},\nu \tilde p_{n\tau}),
\sup_{0\Le n\Le T/\tau}
	\W_2(\nu p_{n\tau},\nu \tilde q_{n\tau}) 
\bigg\}	
\Le C\sqrt{\tau}.
\label{IPS:nu pq}
\end{equation}
\end{corollary}
Note that the LHS of
\eqref{IPS:nu pq} only involves the transition probabilities $\tilde p_{n\tau},\tilde q_{n\tau}$, and does not require the dynamics \eqref{IPS}\eqref{d-IPS}\eqref{d-RB-IPS} to be coupled.
This is because the Wasserstein distance compares the distribution laws rather than trajectories.
\subsection{Uniform-in-time moment estimate}
To investigate the long-time behavior of the numerical methods, we need some preliminary results on the moment estimates. Under appropriate dissipation conditions, it can be proved that the discrete IPS \eqref{d-IPS} and the discrete RB--IPS \eqref{d-RB-IPS} have uniform-in-time moment estimates.
\begin{assumption}
\label{assumption:dissipation}
For the drift force $b:\mathbb R^d\rightarrow\mathbb R^d$, there exist constants $\alpha,\theta>0$ such that
\begin{equation}
	-x\cdot b(x) \Ge \alpha|x|^2-\theta,
	~~~~\forall x\in\mathbb R^d.
	\label{dissipation condition}
\end{equation}
\end{assumption}
The following result is crucial to establish the recurrence relations of both $\mathbb E|\tilde X_n^i|^4$ and $\mathbb E|\tilde Y_n^i|^4$.
\begin{lemma}
\label{lemma:f24}
Under Assumptions \ref{assumption:bounded} and \ref{assumption:dissipation}, let $f(x,\tau) := x + b(x)\tau$ and $\tau_0:=\min\{\alpha/(2L_0^2),1/(2\alpha)\}$.
\begin{enumerate}
\item[$(1)$] There exists a constant $C = C(\alpha,\theta)$ such that if $\tau<\tau_0$,
\begin{equation}
	|f(x,\tau)|^4\Le (1-\alpha\tau)|x|^4+C\tau.
\label{f24:inequality 1}
\end{equation}
\item[$(2)$] For any $\gamma\in\mathbb R^d$ with $|\gamma|\Le L_1$, there exists a constant $C = C(\alpha,\theta,L_1)$ such that if $\tau<\tau_0$,
\begin{equation}
|f(x,\tau)+\gamma\tau|^4 \Le \Big(1-\frac{\alpha\tau}2\Big)|x|^4 + C\tau.
\label{f24:inequality 2}
\end{equation}
\end{enumerate}
\end{lemma}
The proof of Lemma \ref{lemma:f24} is elementary and is in Appendix. 
In Lemma \ref{lemma:f24}, $f(x,\tau) = x + b(x)\tau$ can be viewed as a simplified Euler--Maruyama scheme, where the time step $\tau$ is restricted to be smaller than $\tau_0$ to ensure the stability. In the following, we shall always adopt $\tau_0:=\min\{\alpha/(2L_0^2),1/(2\alpha)\}$ as the upper bound of the time step $\tau$. Note that $\tau_0$ is uniquely determined from Assumptions \ref{assumption:bounded} and \ref{assumption:dissipation} and does not depend on $N$.

Using Lemma \ref{lemma:f24}, we have the following uniform-in-time moments estimates for the discrete IPS \eqref{d-IPS} and the discrete RB--IPS \eqref{d-RB-IPS}.
\begin{theorem}
\label{theorem:moment uniform}
Under Assumptions \ref{assumption:bounded} and \ref{assumption:dissipation},
if there exists a constant $M_4$ such that
$$
	\max_{1\Le i\Le N}\mathbb E|X_0^i|^4\Le M_4,
$$
and if the time step $\tau$ satisfies
$$
	\tau<\min\bigg\{
	\frac{\alpha}{2L_0^2},\frac1{2\alpha}
	\bigg\},
$$
then there exists a constant $C=C(\alpha,\theta,L_1,\sigma)$ such that
\begin{equation}
	\max\bigg\{\sup_{n\Ge0}
	\mathbb E|\tilde X_n^i|^4,
	\sup_{n\Ge0}
		\mathbb E|\tilde Y_n^i|^4\bigg\}
	\Le \max\{M_4,C\},~~~~
	i=1,\cdots,N.
\end{equation}
\end{theorem}
The proof of Theorem \ref{theorem:moment uniform} in Appendix.

Theorem \ref{theorem:moment uniform} tells that when the time step $\tau<\tau_0$, the fourth-order moments of the discrete IPS and the discrete RB--IPS can be bounded uniformly in time.
\begin{remark}
We have some remarks on Theorem \ref{theorem:moment uniform}.
\begin{enumerate}
\item We estimate the fourth-order moments of $\tilde X_n^i$ and $\tilde Y_n^i$ rather than the second-order moments because applying Theorem \ref{theorem:strong error} requires the initial distribution to have finite fourth-order moments.
\item Utilization of the dissipation condition \eqref{dissipation condition} in essential in the proof of Theorem \ref{theorem:moment uniform}. From the geometric perspective, the drift force $b(x)$ pulls the particle $x\in\mathbb R^d$ back when $x$ is far from the origin, hence the particle system shall stay in a bounded region for most of the time, and the moments are bounded uniformly in time. It can also be proved that, if the initial distribution $\nu$ has finite moments of order $2m$ for positive some integer $m\in\mathbb N$, then $\mathbb E|\tilde X_n^i|^{2m}$ and $\mathbb E|\tilde Y_n^i|^{2m}$ are bounded uniformly in time.
\item The constant $C = C(\alpha,\beta,L_1)$ in Theorem \ref{theorem:moment uniform} does not depend on $L_0$, which is related to the boundedness of the drift force $b$. In other words, the moment upper bound is completely controlled by the dissipation condition \eqref{dissipation condition}.
\end{enumerate}
\end{remark}
\subsection{Geometric ergodicity of IPS}
In order to investigate the long-time behavior of the IPS \eqref{IPS} and its mean-field limit, the MVP \eqref{MVP}, it is important that the distribution law of the IPS \eqref{IPS} converges to the equilibrium with a convergence rate $\beta$ independent of the number of particles $N$. If the independence of $\beta$ on $N$ holds true, hopefully the distribution law of the MVP \eqref{MVP} also converges with the convergence rate $\beta$, which allows us to prove the geometric ergodicity of the nonlinear MVP \eqref{MVP}. Therefore, a natural question in studying the ergodicity is to find the conditions ensuring the IPS \eqref{IPS} have a convergence rate independent of $N$.

On the one hand, the interaction force $K$ needs to be moderately large to ensure the uniform-in-$N$ convergence rate.
If the drift force $b$ is not the gradient of a strongly convex function, it is well-known that the MVP \eqref{MVP} can have multiple invariant distributions when the interaction force $K$ is too large, see \cite{uniform_chaos} for example. In this case the IPS \eqref{IPS} must not have a convergence rate independent of $N$.

On the other hand, it is sufficient for the interaction force $K$ to be moderately large to ensure the uniform-in-$N$ convergence rate. To our knowledge, two major approaches to derive the uniform geometric ergodicity of the IPS \eqref{IPS} are the log-Sobolev inequality \cite{log_Sobolev} and the reflection coupling technique \cite{reflect_2,reflect_3}. Under appropriate dissipation conditions, \cite{log_Sobolev} proves the ergodicity in the sense of relative entropy, while \cite{reflect_2,reflect_3} proves the ergodicity in the $\W_1$ distance. Although the relative entropy is stronger than the $\W_1$ distance, in this paper we shall use the $\W_1$ distance because it is compatible with the triangle inequality framework.

In the following, we restate the result of geometric ergodicity of the IPS \eqref{IPS} in the $\W_1$ distance in \cite{reflect_2}.
The dissipation of the drift force $b$ is characterized by a function $\kappa:(0,+\infty)\rightarrow\mathbb R$ satisfying
\begin{equation}
	\kappa(r) \Le \bigg\{
		-\frac2{\sigma^2}
		\frac{(x-y)\cdot(b(x)-b(y))}{|x-y|^2}:x,y\in\mathbb R^d,|x-y|=r
	\bigg\}.
	\label{kappa}
\end{equation}
Assumption \ref{assumption:dissipation} is now replaced by the asymptotic positivity of $\kappa(r)$.
\begin{assumption}
\label{assumption:kappa}
The function $\kappa(r)$ defined in \eqref{kappa} satisfies
\begin{enumerate}
\setzero
\item[$1.$] $\kappa(r)$ is continuous for $r\in(0,+\infty)$;
\item[$2.$] $\kappa(r)$ has a lower bound for $r\in(0,+\infty)$;
\item[$3.$] $\di\varlimsup_{r\rightarrow\infty} \kappa(r)>0$.
\end{enumerate}
\end{assumption}
We note that Assumption \ref{assumption:kappa} is stronger than Assumption \ref{assumption:dissipation}. In fact, the asymptotic positivity of $\kappa(r)$ implies that there exist positive constants $\alpha,\beta>0$ such that
\begin{equation}
	r^2\kappa(r) \Ge \alpha r^2-\beta,~~~~\forall r>0.
\end{equation}
Then we easily obtain
\begin{equation}
	-(x-y)\cdot (b(x) - b(y)) \Le \frac{\sigma^2}2(\alpha|x-y|^2-\beta), 
\end{equation}
and thus \eqref{dissipation condition} holds. 
Under Assumption \ref{assumption:kappa}, we can construct a concave function $f:[0,+\infty)\rightarrow[0,+\infty)$ satisfying the following.
\begin{lemma}
\label{lemma:distance}
Under Assumption \ref{assumption:kappa}, there exists a function $f:[0,+\infty)\rightarrow[0,+\infty)$ satisfying
\begin{enumerate}
\setzero
\item[$1.$] $f(0) = 0$, and $f(r)$ is concave and strictly increasing in $r\in[0,+\infty)$;
\item[$2.$] $f\in C^2[0,+\infty)$ and there exists a constant $c_0>0$ such that
\begin{equation}
	f''(r) - \frac14r\kappa(r)f'(r) 
	\Le - \frac{c_0}2 f(r),~~~~
	\forall r\Ge0.
	\label{f_inequality}
\end{equation}
\item[$3.$] There exists a constant $\vp_0>0$ such that
\begin{equation}
	\frac{\vp_0}4r\Le f(r) \Le r.
\end{equation}
\end{enumerate}
\end{lemma}
The proof of Lemma \ref{lemma:distance} can be seen in Theorem 1 of \cite{reflect_2} or Lemma 2.1 in \cite{rbme_2}. Although Lemma \ref{lemma:distance} serves as part of the proof of the geometric ergodicity for the IPS \eqref{IPS} and is not directly related to the topic of this paper, it does provide an explicit upper bound of the constant $L_1$ in \eqref{boundedness 1}, which is used in the statement of the main theorem.

Define the space of probability distributions with finite first-order moments by
\begin{equation}
	\P_1(\mathbb R^{Nd}) = \bigg\{
		\nu\in\P(\mathbb R^{Nd}):
		\max_{1\Le i\Le N}\int_{\mathbb R^{Nd}}
		|x^i|\nu(\d x) <+\infty
	\bigg\}.
\end{equation}
Equipped with the normalized Wasserstein-1 distance, $(\P_1(\mathbb R^{Nd}),\W_1)$ is a complete metric space. Now we have the following result of geometric ergodicity for the IPS \eqref{IPS}.
\begin{theorem}
\label{theorem:IPS ergodicity}
Under Assumptions \ref{assumption:bounded} and \ref{assumption:kappa}, if the constant $L_1$ in \eqref{boundedness 1} satisfies
$$
	L_1 < \frac{c_0\vp_0\sigma^2}{16},
$$
then for $\beta:=c_0\sigma^2/2$ there exists a positive constant $C = C(\kappa,\sigma)$ such that
\begin{equation}
	\W_1(\mu p_t,\nu p_t) \Le Ce^{-\beta t} \W_1(\mu,\nu),~~~~
	\forall t\Ge0
	\label{IPS W1 ergodicity}
\end{equation}
for any probability distributions $\mu,\nu\in\P_1(\mathbb R^{Nd})$.
\end{theorem}
The proof of Theorem \ref{theorem:IPS ergodicity} can be seen at Corollary 2 in \cite{reflect_2} or Theorem 2.2 in \cite{rbme_2}.
\begin{remark}
We have some remarks on Theorem \ref{theorem:IPS ergodicity}.
\begin{enumerate}
\item Using the reflection coupling technique, we can actually prove that the IPS \eqref{IPS} is contractive in the Wasserstein-$f$ distance:
\begin{equation}
	\W_f(\mu p_t,\nu p_t) \Le e^{-\beta t} \W_f(\mu,\nu),
\end{equation}
where $\W_f(\cdot,\cdot)$ is the normalized Wasserstein-1 distance induced by the function $f$,
\begin{equation}
	\W_f(\mu,\nu) = \inf_{\gamma\in\Pi(\mu,\nu)}
	\int_{\mathbb R^{Nd}\times\mathbb R^{Nd}}
	\bigg(
	\frac1N\sum_{i=1}^Nf(|x^i-y^i|)
	\bigg)\gamma(\d x\d y).
	\label{IPS Wf ergodicity}
\end{equation}
Since $f(r)$ is equivalent to the usual Euclidean norm, \eqref{IPS W1 ergodicity} is a direct consequence of \eqref{IPS Wf ergodicity}.
\item The explicit convergence rate $\beta=c_0\sigma^2/2$ and the upper bound $c_0\vp_0\sigma^2/16$ only depend on $\kappa(r)$ and $\sigma$. In particular, these parameters do not depend on the number of particles $N$. Hence the IPS \eqref{IPS} has an exponential convergence rate independent of $N$.
\item The positivity of the diffusion constant $\sigma$ is essential in the proof by reflection coupling. In fact, for given interaction force $K$, the MVP \eqref{MVP} can be non-ergodic if $\sigma$ is too small \cite{small_diffusion}.
\end{enumerate}
\end{remark}
Using Theorem \ref{theorem:IPS ergodicity}, the existence and uniqueness of the invariant distribution $\pi\in\P(\mathbb R^{Nd})$ can be derived using the Banach fixed point theorem.
\begin{corollary}
\label{corollary:IPS ergodicity invariant}
Under Assumptions \ref{assumption:bounded} and \ref{assumption:kappa}, if the constant $L_1$ in \eqref{boundedness 1} satisfies
$$
	L_1 < \frac{c_0\vp_0\sigma^2}{16},
$$
then the IPS \eqref{IPS} has a unique invariant distribution $\pi\in\P_1(\mathbb R^{Nd})$, and for $\beta:=c_0\sigma^2/2$ there exist a positive constant $C = C(\kappa,\sigma)$ such that
\begin{equation}
	\W_1(\nu p_t,\pi) \Le Ce^{-\beta t} \W_1(\nu,\pi),~~~~
	\forall t\Ge0
\end{equation}
for any probability distributions $\nu\in\P_1(\mathbb R^{Nd})$.
\end{corollary}
The proof of Corollary \ref{corollary:IPS ergodicity invariant} can be seen at Corollary 3 in \cite{reflect_2} or Theorem 3.1 in \cite{rbme_2}.
\subsection{Wasserstein-1 error in long time}
We estimate the long-time sampling error of the discrete IPS \eqref{d-IPS} and the discrete RB--IPS \eqref{d-RB-IPS} in the $\W_1$ distance using the triangle inequality and results in previous subsections. We begin with the following induction lemma, which can be viewed as a discrete version of Lemma \ref{lemma:general}.
\begin{lemma}
\label{lemma:induction}
Given $m\in\mathbb N$, $\ep>0$ and $q\in(0,1)$. If a nonnegative sequence $\{a_n\}_{n\Ge0}$ satisfies
\begin{equation}
	a_n \Le \ep + qa_{n-m},~~~~\forall n\Ge m,
	\label{induction recurrence}
\end{equation}
then
\begin{equation}
	a_n \Le \frac{\ep}{1-q} + Mq^{\frac nm-1},~~~~
	\forall n\Ge0,
	\label{induction an}
\end{equation}
where $M=\di\max_{0\Le k\Le m-1} a_k$.
\end{lemma}
The proof of Lemma \ref{lemma:induction} is in Appendix. Lemma \ref{lemma:induction} implies that if $q<1$ in the recurrence relation \eqref{induction recurrence}, then the asymptotic form of $a_n$ is $O(\ep)$ plus an exponential tail.

Combining the finite-time difference relation \eqref{IPS:nu pq} in Corollary \ref{corollary:IPS finite error} and the geometric ergodicity in Theorem \ref{theorem:IPS ergodicity}, we employ the triangle inequality to estimate the long-time sampling error of the numerical methods \eqref{d-IPS}\eqref{d-RB-IPS}. Recall that the transition probabilities of the dynamics \eqref{d-IPS}\eqref{d-RB-IPS} are $\tilde p_{n\tau},\tilde q_{n\tau}$ respectively.
\begin{theorem}
\label{theorem:IPS sample quality}
Under Assumptions \ref{assumption:bounded} and \ref{assumption:kappa}, if there exists a constant $M_4$ such that
$$
	\max_{1\Le i\Le N}\int_{\mathbb R^{Nd}}
	|x^i|^4\nu(\d x)
	\Le M_4,
$$
and the constant $L_1$ in \eqref{boundedness 1} and the time step $\tau$ satisfy
$$
	L_1 < \frac{c_0\vp_0\sigma^2}{16},~~~~
	\tau < \min\bigg\{
	\frac{\alpha}{2L_0^2},\frac1{2\alpha}
	\bigg\},
$$
then there exist positive constants $\lambda=\lambda(\kappa,L_0,\sigma)$ and $C = C(\kappa,L_0,M_4,\sigma)$ such that
\begin{enumerate}
\item[$1.$] The transition probability $(\tilde p_{n\tau})_{n\Ge0}$ of discrete IPS \eqref{d-IPS} satisfies
\begin{equation}
	\W_1(\nu \tilde p_{n\tau},\pi) \Le C\sqrt{\tau} + Ce^{-\lambda n\tau},~~~~
	\forall n\Ge0.
	\label{IPS long time error 1}
\end{equation}
\item[$2.$] The transition probability $(\tilde q_{n\tau})_{n\Ge0}$ of discrete RB--IPS \eqref{d-RB-IPS} satisfies
\begin{equation}
	\W_1(\nu \tilde q_{n\tau},\pi) \Le C\sqrt{\tau} + Ce^{-\lambda n\tau},~~~~
	\forall n\Ge0.
	\label{IPS long time error 2}
\end{equation}
\end{enumerate}
\end{theorem}
\begin{proof}
For any given integers $n\Ge m$, we have the following triangle inequality
\begin{equation}
	\W_1(\nu \tilde p_{n\tau},\pi) \Le \W_1(\nu \tilde p_{(n-m)\tau}\tilde p_{m\tau},\nu\tilde p_{(n-m)\tau}p_{m\tau}) + 
	\W_1(\nu_0\tilde p_{(n-m)\tau}p_{m\tau},\pi p_{m\tau}).
\label{IPS proof triangle}
\end{equation}
By Theorem \ref{theorem:moment uniform}, $\nu\tilde p_{(n-m)\tau}$ has uniform-in-time fourth order moment estimates, i.e., there exists a constant $M_4' = M_4'(\kappa,M_4,\sigma)$ such that
\begin{equation}
	\max_{1\Le i\Le N}\bigg\{
	\sup_{n\Ge m}
	\int_{\mathbb R^{Nd}}|x^i|^4
	\big(\nu\tilde p_{(n-m)\tau}\big)
	(\d x) \bigg\}
	\Le M_4'.
\end{equation}
Hence by Corollary \ref{corollary:IPS finite error}, there exists a constant $C_1 = C_1(\kappa,L_0,M_4,m\tau,\sigma)$ such that
\begin{equation}
	\W_1(\nu \tilde p_{(n-m)\tau}\tilde p_{m\tau},\nu\tilde p_{(n-m)\tau}p_{m\tau}) \Le C_1\sqrt{\tau},~~~~
	\forall n\Ge m.
	\label{IPS proof finite time}
\end{equation}
The constant $C_1$ depends on the upper bound of $m\tau$, which is the evolution time of the IPS \eqref{IPS} and the discrete IPS \eqref{d-IPS}.
By Theorem \ref{theorem:IPS ergodicity}, there exists a constant $C_0 = C_0(\kappa,\sigma)$ such that
\begin{equation}
	\W_1(\nu\tilde p_{(n-m)\tau}p_{m\tau},\pi p_{m\tau}) \Le C_0e^{-\beta m\tau}
\W_1(\nu \tilde p_{(n-m)\tau},\pi),~~~~
\forall n\Ge m.
\label{IPS proof ergodicity}
\end{equation}
From \eqref{IPS proof triangle}\eqref{IPS proof finite time}\eqref{IPS proof ergodicity} we obtain
\begin{equation}
	\W_1(\nu \tilde p_{n\tau},\pi) \Le C_1\sqrt{\tau} + C_0e^{-\beta m\tau} \W_1(\nu_0 \tilde p_{(n-m)\tau},\pi),~~~~
	\forall n\Ge m.
	\label{IPS proof long time}
\end{equation}
For given time step $\tau>0$, we wish to choose $m$ to satisfy $C_0e^{-\beta m\tau}=1/e$, so that Lemma \ref{lemma:induction} can be applied. However, $m$ is restricted to be an integer, thus our choice is
\begin{equation}
	m =
	\left\lceil
	\frac{\log C_0+1}{\beta\tau}\right\rceil.
\end{equation}
It is easy to check $m\tau$ has an upper bound independent of $\tau$,
\begin{equation}
	m\tau \Le \bigg(
	\frac{\log C_0+1}{\beta\tau} + 1
	\bigg)\tau \Le \frac{\log C_0+1}{\beta} + \frac1{2\alpha},
\end{equation}
hence the constant $C_1$ in \eqref{IPS proof finite time} can be made independent of $\tau$, i.e., $C_1 = C_1(\kappa,L_0,M_4,\sigma)$. Note that for this choice of $m$ we have $C_0 e^{-\beta m\tau}\Le 1/e$, and \eqref{IPS proof long time} implies
\begin{equation}
	\W_1(\nu\tilde p_{n\tau},\pi) 
	\Le C_1\sqrt{\tau} + \frac1e\W_1(\nu_0 \tilde p_{(n-m)\tau},\pi),~~~~
	\forall n\Ge m.
\end{equation}
Applying Lemma \ref{lemma:induction} with $a_n:=\W_1(\nu_0 \tilde p_{n\tau},\pi)$, we have
\begin{equation}
	\W_1(\nu\tilde p_{n\tau},\pi) \Le
	C_1\sqrt{\tau} + M_0 e^{1-\frac nm},~~~~
	\forall n\Ge0,
	\label{IPS proof final 1}
\end{equation}
where the constant 
\begin{equation}
	M_0 := \sup_{0\Le k\Le m-1}
	\W_1(\nu\tilde p_{k\tau},\pi) \Le 
	\sup_{k\Ge0}
	\W_1(\nu\tilde p_{k\tau},\pi).
	\label{IPS proof final 2}
\end{equation}
Introduce the normalized moment for $\nu\in\P_1(\mathbb R^{Nd})$ by
\begin{equation}
	\M_1(\nu) = \int_{\mathbb R^{Nd}}
	\bigg(\frac1N\sum_{i=1}^N |x^i|\bigg) \nu(\d x),
\end{equation}
then the $\W_1$ distance is bounded by 
\begin{equation}
	\W_1(\nu\tilde p_{k\tau},\pi) \Le \M_1(\nu\tilde p_{k\tau}) + \M_1(\pi).
	\label{IPS proof final 3}
\end{equation}
On the one hand, $\nu\tilde p_{k\tau}$ has uniform-in-time  fourth-order moments, hence there exists a constant $C_2 = C_2(\kappa,L_0,M_4,\sigma)$ such that
\begin{equation}
	\sup_{k\Ge0} \M_1(\nu\tilde p_{k\tau}) \Le C_2.
	\label{IPS proof final 4}
\end{equation}
On the other hand, by Lemma 3.1 in \cite{rbme_2}, for the invariant distribution $\pi$ of the IPS \eqref{IPS}, there exists a constant $C_2 = C_2(\kappa,L_0,M_4,\sigma)$ such that
\begin{equation}
	\M_1(\pi) \Le C_2.
	\label{IPS proof final 5}
\end{equation}
Combining \eqref{IPS proof final 1}--\eqref{IPS proof final 5}, we obtain
\begin{equation}
\W_1(\nu\tilde p_{n\tau},\pi) \Le
	C_1\sqrt{\tau} + C_2 e^{-\frac nm},~~~~
	\forall n\Ge0,
\label{IPS proof final}
\end{equation}
where both constants $C_1,C_2$ only depend on $\kappa,L_0,M_4,\sigma$. Note that by the choice of $m$
\begin{equation}
	\frac{n}{m} \Ge \frac{n}{\frac{\log C_0+1}{c\tau}+1} \Ge \frac{\beta n\tau}{\log C_0+\beta/(2\alpha)+1},
\end{equation}
hence by defining
\begin{equation}
	\lambda:=\frac{\beta}{\log C_0+\beta/(2\alpha)+1},
\end{equation}
there holds $e^{-n/m} \Le e^{-\lambda n\tau}$. Hence \eqref{IPS proof final} implies
\begin{equation}
	\W_1(\nu \tilde p_{n\tau},\pi) \Le C\sqrt{\tau} + Ce^{-\lambda n\tau},~~~~\forall n\Ge0,
\end{equation}
which is exactly the long-time sampling error. The proof for the discrete RB--IPS is the same.
\end{proof}
Theorem \ref{theorem:IPS sample quality} produces the long-time sampling error of the two numerical methods, the discrete IPS \eqref{d-IPS} and the discrete RB--IPS \eqref{d-RB-IPS}, in the $\W_1$ distance.
The error in \eqref{IPS long time error 1}\eqref{IPS long time error 2} consists of two parts:
$C\sqrt{\tau}$ represents the bias between the invariant distribution $\pi$ and the asymptotic limit of $\nu \tilde p_{n\tau}$ or $\nu \tilde q_{n\tau}$, while $Ce^{-\lambda n\tau}$ represents the exponential convergence of the numerical methods. Here the convergence rate $\lambda = \lambda(\kappa,L_0,\sigma)$ can be different from the convergence rate $\beta:=c_0\sigma^2/2$ of the IPS \eqref{IPS}. Still, $\lambda$ is independent of the number of particles $N$, the time step $\tau$, the batch size $p$ and the choice of the initial distribution $\nu$.
\begin{remark}
We have some remarks on Theorem \ref{theorem:IPS sample quality}.
\begin{enumerate}
\item Assumption \ref{assumption:dissipation} is a corollary of Assumption \ref{assumption:kappa}, and the constants $\alpha,\theta$ in \eqref{dissipation condition} can be directly derived from Assumption \ref{assumption:kappa}.
\item The constant $C$ in \eqref{IPS long time error 2} depends on $M_4$, the fourth-order moments of initial distribution $\nu$. However, the convergence rate $\lambda$ does not depend on $M_4$. In practical simulation, one may choose the initial distribution as the Dirac distribution centered at origin to sample the invariant distribution $\pi$, and in this case the dependence of $C$ on $M_4$ can be ignored.
\item Since we are studying the long-time behavior of the numerical methods,
it is natural to ask: do the numerical methods \eqref{d-IPS}\eqref{d-RB-IPS} have invariant distributions? If so, does the convergence rate depends on $N$? The existence of the invariant distributions can be proved by the Harris ergodic theorem under appropriate Lyapunov conditions, see \cite{tri_2,var_1,var_2}. However, the convergence rate derived from the Harris ergodic theorem is very implicit. Still, there are a few results which proved that the convergence rate of the numerical method can be independent of $N$, under global boundedness condition of the drift force $b$ \cite{mixing}, which are too strong for practical use. In this paper we follow the idea in \cite{tri_4} and avoid discussing the geometric ergodicity of the numerical methods themselves.
\end{enumerate}
\end{remark}

\section{Error analysis of discrete RB--IPS for MVP}
In this section we analyze the error of the discrete RB--IPS \eqref{RB-IPS}, as a numerical approximation to the MVP \eqref{MVP}. Thanks to the theory of propagation of chaos, we can easily extend our results in Section 3 for the IPS \eqref{IPS} to the case of the MVP \eqref{MVP}. Nevertheless, we should be careful that the major difference between the IPS \eqref{IPS} and the MVP \eqref{MVP} is the \emph{linearity} of the transition probability, as we illustrate follows.

The transition probability $(p_t)_{t\Ge0}$ of the IPS \eqref{IPS} forms a linear semigroup, that is,
\begin{enumerate}
\setzero
\item For any $\nu\in\P(\mathbb R^{Nd})$, $(\nu p_t)p_s = \nu p_{t+s}$;
\item For any $t>0$, the mapping $\nu\mapsto \nu p_t$ is linear in $\nu\in\P(\mathbb R^{Nd})$.
\end{enumerate}
Denote the transition probability of the MVP \eqref{MVP} by $(\bar p_t)_{t\Ge0}$, then for any $\nu\in\P(\mathbb R^d)$, $\nu \bar p_t$ is the distribution law of $\bar X_t$ in the MVP \eqref{MVP}. Although $(\bar p_t)_{t\Ge0}$ still satisfies the
semigroup property $(\nu \bar p_t)\bar p_s = \nu \bar p_{t+s}$, $(\bar p_t)_{t\Ge0}$ does not form a linear semigroup, because the MVP \eqref{MVP} is a distribution dependent SDE and thus the mapping $\nu\mapsto\nu \bar p_t$ is nonlinear. The readers may also refer to \cite{ddsde} for a complete guide to distribution-dependent SDEs and nonlinear semigroups.
\subsection{Strong error in finite time}
To estimate the strong error between the IPS \eqref{IPS} in $\mathbb R^{Nd}$ and the MVP \eqref{MVP} in $\mathbb R^d$, we need to define the synchronous coupling between \eqref{IPS}\eqref{MVP}. Given the initial distribution $\nu\in\P(\mathbb R^d)$ and $N$ independent Wiener processes $\{W_t^i\}_{i=1}^N$, the strong solution to the IPS \eqref{IPS} is
\begin{equation}
	X_t^i = X_0^i + 
	\int_0^t
	\bigg(b(X_s^i) + \frac1{N-1}
	\sum_{j\neq i} K(X_s^i - X_s^j)\bigg)\d s + \sigma W_t^i,~~~~
	i=1,\cdots,N,
	\label{strong solution IPS}
\end{equation}
where the initial value $\{X_0^i\}_{i=1}^N$ are sampled from $\nu$ independently. Introduce $N$ duplicates of the MVP \eqref{MVP} represented by $\{\bar X_t^i\}_{i=1}^N$, where each $\bar X_t^i$ is the strong solution to the SDE
\begin{equation}
	\bar X_t^i = X_0^i + 
	\int_0^t \bigg(
		b(\bar X_s^i) + 
		(K*\Law(\bar X_s^i))(\bar X_s^i)
	\bigg)\d s + \sigma\d W_t^i,~~~~
	i=1,\cdots,N.
	\label{strong solution MVP}
\end{equation}
Here, `$*$' denotes the convolution of a density kernel with a probability distribution,
\begin{equation}
	(K*\mu)(x) = \int_{\mathbb R^d} K(x-z) \mu(\d z).
\end{equation}
It can be observed from \eqref{strong solution IPS}\eqref{strong solution MVP} that each $\bar X_t^i$ uses the same initial value $X_0^i\sim \nu$ and the same Wiener process $W_t^i$ with $X_t^i$. The major difference between \eqref{strong solution IPS}\eqref{strong solution MVP} is that the particles in $\{X_t^i\}_{i=1}^N$ are interacting with each other, while the particles in $\{\bar X_t^i\}_{i=1}^N$ are fully decoupled, i.e., the evolution of the $N$ particles in $\{\bar X_t^i\}_{i=1}^N$ is mutually independent.

The estimate of the strong error between \eqref{strong solution IPS}\eqref{strong solution MVP} is a classical topic in the theory of propagation of chaos. The first known result was derived by McKean \cite{McKean} and is stated as follows.
\begin{theorem}
\label{theorem:McKean}
Under Assumption \ref{assumption:bounded}, there exists a constant $C = C(L_0,L_1,T,\sigma)$ such that
\begin{equation}
	\frac1N\sum_{i=1}^N
	\mathbb E\Big[\sup_{t\Le T}
	|X_t^i - \bar X_t^i|^2
	\Big] \Le \frac{C}{N}.
\end{equation}
\end{theorem}
As in the synchronous coupling, the expectation is taken over the Wiener processes $\{W_t^i\}_{i=1}^N$ in the time interval $[0,T]$ and the random variables $\{X_0^i\}_{i=1}^N$.
We note that the IPS \eqref{IPS} in this paper is slightly different from the original setting in \cite{McKean}, where the perturbation force $\gamma^i(x)$ is given by
\begin{equation}
	\gamma^i(x) = \frac1N\sum_{j=1}^N 
	K(x^i - x^j)
	\label{McKean setting 1}
\end{equation}
rather than
\begin{equation}
	\gamma^i(x) = \frac1{N-1}\sum_{j\neq i}
	K(x^i - x^j).
	\label{McKean setting 2}
\end{equation}
This minor difference in the choice of $\gamma^i$ does not impact the final result of propagation of chaos. The proof of Theorem \ref{theorem:McKean} under the settings \eqref{McKean setting 1}\eqref{McKean setting 2} can be found in Theorem 3.1 of \cite{chaos_3} and Proposition 4.2 in \cite{rbme_3} respectively.

Combining Theorems \ref{theorem:IPS strong} and \ref{theorem:McKean}, we directly obtain the strong error of the discrete IPS \eqref{d-IPS}.
\begin{corollary}
\label{corollary:IPS strong}
Under Assumption \ref{assumption:bounded}, if there exists a constant $M_2$ such that
$$
	\int_{\mathbb R^d}
	|x|^2\nu(\d x)
	\Le M_2,
$$
then there exist constants $C_1=C_1(L_0,L_1,M_2,T,\sigma)$ and $C_2 = C_2(L_0,L_1,T,\sigma)$ such that
\begin{equation}
	\sup_{0\Le n\Le T/\tau}\bigg(
	\frac1N\sum_{i=1}^N
	\mathbb E|\bar X^i_{n\tau} - \tilde X^i_n|^2\bigg) \Le C_1\tau + \frac{C_2}{N}.
	\label{MVP strong error 1}
\end{equation}
\end{corollary}
Combining Theorems \ref{theorem:strong error} and \ref{theorem:McKean}, we obtain the strong error of the discrete RB--IPS \eqref{d-RB-IPS}.
\begin{corollary}
\label{corollary:RB-IPS strong}
Under Assumption \ref{assumption:bounded}, if there exists a constant $M_4$ such that
$$
	\int_{\mathbb R^d}
	|x|^4\nu(\d x)
	\Le M_4,
$$
then there exist constants $C_1=C_1(L_0,L_1,M_4,T,\sigma)$ and $C_2 = C_2(L_0,L_1,T,\sigma)$ such that
\begin{equation}
	\sup_{0\Le n\Le T/\tau}\bigg(
	\frac1N\sum_{i=1}^N
	\mathbb E|\bar X^i_{n\tau} - \tilde Y^i_n|^2\bigg) \Le C_1\tau + \frac{C_2}{N}.
	\label{MVP strong error 2}
\end{equation}
\end{corollary}

In the $\W_2$ distance, the finite-time error of the discrete IPS \eqref{d-IPS} and the discrete RB--IPS \eqref{d-RB-IPS} is estimated as follows.
\begin{corollary}
\label{corollary:MVP finite error}
Under Assumption \ref{assumption:bounded}, if there exists a constant $M_4$ such that
$$
	\int_{\mathbb R^d} |x|^4\nu(\d x)\Le M_4,
$$
then there exists constant $C_1=C_1(L_0,L_1,M_4,T,\sigma)$ and $C_2 = C_2(L_0,L_1,T,\sigma)$ such that
\begin{equation}
\max\bigg\{
\sup_{0\Le n\Le T/\tau}
\W_2(\nu^{\otimes N} \bar p_{n\tau}^{\otimes N},
\nu^{\otimes N} \tilde p_{n\tau}),
\sup_{0\Le n\Le T/\tau}
\W_2(\nu^{\otimes N} \bar p_{n\tau}^{\otimes N},
\nu^{\otimes N} \tilde q_{n\tau}) 
\bigg\}	
\Le C_1\sqrt{\tau} + \frac{C_2}{\sqrt{N}}.
\end{equation}
\end{corollary}
Here, $\nu^{\otimes N}\in\P(\mathbb R^{Nd})$ denotes the tensor product of the distribution $\nu\in\P(\mathbb R^d)$, and $\bar p_t^{\otimes N}$ denotes the product of $\bar p_t$ in $\mathbb R^{Nd}$. Recall that the $N$ duplicates $\{\bar X_t^i\}_{i=1}^N$ of the MVP \eqref{MVP} are mutually independent, hence $\nu^{\otimes N} \bar p_{n\tau}^{\otimes N} = (\nu \bar p_{n\tau})^{\otimes N}$.

Let $[\mu]_1\in\P(\mathbb R^d)$ denote the marginal distribution of a symmetric distribution $\mu\in\P(\mathbb R^{Nd})$ (see Definition 2.1 in \cite{chaos_3}). Note that \eqref{MVP strong error 1}\eqref{MVP strong error 2} can be written as
\begin{equation}
	\max\bigg\{
	\sup_{0\Le n\Le T/\tau}
	\mathbb E|\bar X_{n\tau}^1 - \tilde X_n^1|^2
	\sup_{0\Le n\Le T/\tau}
	\mathbb E|\bar X_{n\tau}^1 - \tilde Y_n^1|^2
	\bigg\}
	\Le C_1\tau + \frac{C_2}{N},
\end{equation}
hence in the sense of marginal distributions we have the following.
\begin{corollary}
\label{corollary:MVP finite error marginal}
Under Assumption \ref{assumption:bounded}, if there exists a constant $M_4$ such that
$$
	\int_{\mathbb R^d} |x|^4\nu(\d x)\Le M_4,
$$
then there exist constants $C_1=C_1(L_0,L_1,M_4,T,\sigma)$ and $C_2 = C_2(L_0,L_1,T,\sigma)$ such that
\begin{equation}
\max\bigg\{
\sup_{0\Le n\Le T/\tau}
\W_2\big(\nu \bar p_{n\tau},
[\nu^{\otimes N} \tilde p_{n\tau}]_1\big),
\sup_{0\Le n\Le T/\tau}
\W_2\big(\nu \bar p_{n\tau},
[\nu^{\otimes N} \tilde q_{n\tau}]_1\big) 
\bigg\}	
\Le C_1\sqrt{\tau} + \frac{C_2}{\sqrt{N}}.
\end{equation}
\end{corollary}
In Corollary \ref{corollary:MVP finite error marginal}, $\nu \bar p_{n\tau}$ is the distribution law of the MVP \eqref{MVP} and does not depend on $N$. Hence Corollary \ref{corollary:MVP finite error marginal} implies that we can obtain the correct distribution law $\nu\bar p_{n\tau} = \Law(\bar X_{n\tau})$ by choosing $N$ sufficiently large and $\tau$ sufficiently small.
\subsection{Geometric ergodicity of MVP}
It has been proved that when the interaction force $K$ is moderately large, the IPS \eqref{IPS} has a convergence rate $\beta$ uniform in the number of particles $N$.
Since the MVP \eqref{MVP} is the mean-field limit of the IPS \eqref{IPS}, it is natural to expect that the MVP \eqref{MVP} also has the convergence rate $\beta$. In fact, the geometric ergodicity of the MVP \eqref{MVP} can be directly from Theorem \ref{theorem:IPS ergodicity}.
\begin{theorem}
\label{theorem:MVP ergodicity}
Under Assumptions \ref{assumption:bounded} and \ref{assumption:kappa}, if the constant $L_1$ in \eqref{boundedness 1} satisfies
$$
	L_1 < \frac{c_0\vp_0\sigma^2}{16},
$$
then for $\beta:=c_0\sigma^2/2$ there exist a positive constant $C = C(\kappa,\sigma)$ such that
\begin{equation}
	\W_1(\mu \bar p_t,\nu \bar p_t) \Le Ce^{-\beta t} \W_1(\mu,\nu),~~~~
	\forall t\Ge0
\end{equation}
for any probability distributions $\mu,\nu\in\P_1(\mathbb R^d)$.
\end{theorem}
The proof of Theorem \ref{theorem:MVP ergodicity} is Appendix. As a consequence, we deduce that the MVP \eqref{MVP} has a unique invariant distribution $\bar\pi\in\P_1(\mathbb R^d)$. Also, the $\W_1$ distance between $\pi$ and $\bar\pi$ can be controlled.
\begin{corollary}
\label{corollary:MVP ergodicity invariant}
Under Assumptions \ref{assumption:bounded} and \ref{assumption:kappa}, if the constant $L_1$ in \eqref{boundedness 1} satisfies
$$
	L_1 < \frac{c_0\vp_0\sigma^2}{16},
$$
then the invariant distribution $\bar\pi\in\P_1(\mathbb R^d)$ of the MVP \eqref{MVP} is unique, and for $\beta:=c_0\sigma^2/2$ there exists a positive constant $C = C(\kappa,\sigma)$ such that
\begin{equation}
	\W_1(\nu \bar p_t,\bar\pi) \Le Ce^{-\beta t} \W_1(\nu,\bar\pi),~~~~
	\forall t\Ge0.
\end{equation}
for any $\nu\in\P_1(\mathbb R^d)$. Furthermore, there exists a constant $C = C(\kappa,L_0,\sigma)$ such that
\begin{equation}
	\W_1(\bar\pi^{\otimes N},\pi) \Le \frac C{\sqrt{N}}.
	\label{W1 uniform-in-time}
\end{equation}
\end{corollary}
The proof of Corollary \ref{corollary:MVP ergodicity invariant} is in Appendix.
We note that \eqref{W1 uniform-in-time} can also viewed as the corollary of the uniform-in-time propagation of chaos, see Theorem 2 in \cite{elementary} for example.
\subsection{Wasserstein-1 error in long time}
Combining Theorem \ref{theorem:IPS sample quality} and Corollary \ref{corollary:MVP ergodicity invariant}, we immediately obtain the following result of long-time sampling error of the discrete IPS \eqref{d-IPS} and the discrete RB--IPS \eqref{d-RB-IPS}.
\begin{theorem}
\label{theorem:MVP sample quality}
Under Assumptions \ref{assumption:bounded} and \ref{assumption:kappa}, if there exists a constant $M_4$ such that
$$
	\max_{1\Le i\Le N}\int_{\mathbb R^{Nd}} 
	|x^i|^4\nu(\d x) \Le M_4,
$$
and the constant $L_1$ in \eqref{boundedness 1} and the time step $\tau$ satisfy
$$
	L_1 < \frac{c_0\vp_0\sigma^2}{16},~~~~
	\tau<\min\bigg\{
	\frac{\alpha}{2L_0^2},\frac1{2\alpha}
	\bigg\},
$$
then there exist positive constants $\lambda = \lambda(\kappa,L_0,\sigma)$, $C_1 = C_1(\kappa,L_0,M_4,\sigma)$ and $C_2 = C_2(\kappa,L_0,\sigma)$
\begin{enumerate}
\item[$1.$] The transition probability $(\tilde p_{n\tau})_{n\Ge0}$ of discrete IPS \eqref{d-IPS} satisfies
\begin{equation}
\W_1(\nu \tilde p_{n\tau},\bar\pi^{\otimes N}) \Le C_1\sqrt{\tau} + C_1e^{-\lambda n\tau} + \frac{C_2}{\sqrt{N}},~~~~
	\forall n\Ge0.
\end{equation}
\item[$2.$] The transition probability $(\tilde q_{n\tau})_{n\Ge0}$ of discrete RB--IPS \eqref{d-RB-IPS} satisfies
\begin{equation}
\W_1(\nu \tilde q_{n\tau},\bar\pi^{\otimes N}) \Le C_1\sqrt{\tau} + C_1e^{-\lambda n\tau} + \frac{C_2}{\sqrt{N}},~~~~
	\forall n\Ge0.
\end{equation}
\end{enumerate}
\end{theorem}

Using the theory of the propgation of chaos, we may translate the normalized Wasserstein-1 distance in $\P(\mathbb R^{Nd})$ to the Wasserstein-1 distance in $\P(\P(\mathbb R^d))$. Denote the empirical distributions of the discrete IPS \eqref{d-IPS} and the discrete RB--IPS \eqref{d-RB-IPS} by $\tilde\mu_{n\tau}\in\P(\mathbb R^d)$ and $\tilde\mu_{n\tau}^{\RB}\in\P(\mathbb R^d)$, i.e.,
\begin{equation}
	\tilde\mu_{n\tau}(x) = \frac1N\sum_{i=1}^N 
	\delta(x-\tilde X_{n\tau}^i)\in\P(\mathbb R^d),~~~~
	\tilde 
	\mu_{n\tau}^{\RB}(x) = 
	\frac1N\sum_{i=1}^N \delta(x-\tilde Y_{n\tau}^i)\in\P(\mathbb R^d).
\end{equation}
Since $\tilde X_{n\tau},\tilde Y_{n\tau}$ are 
random variables with distribution laws $\nu \tilde p_{n\tau},\nu \tilde q_{n\tau}$, the empirical distributions $\tilde\mu_{n\tau},\tilde\mu_{n\tau}^{\RB}$ are actually random measures on $\mathbb R^d$, and thus their distribution laws $\Law(\tilde\mu_{n\tau}),\Law(\tilde\mu_{n\tau}^{\RB})$ can be identified as elements of $\P(\P(\mathbb R^d))$.
By Proposition 2.14 of \cite{chaos_1}, we have
\begin{equation}
	\W_1(\nu\tilde p_{n\tau}, \bar \pi^{\otimes N}) = 
	\W_1(\Law(\tilde \mu_{n\tau}),\delta_{\bar\pi}),~~~~
	\W_1(\nu\tilde q_{n\tau}, \bar \pi^{\otimes N}) = 
	\W_1(\Law(\tilde \mu_{n\tau}^{\RB}),\delta_{\bar\pi}).
\end{equation}
Here, $\W_1$ on the LHS is the normalized Wasserstein-1 distance in $\P(\mathbb R^{Nd})$ defined in \eqref{Wasserstein}, and $\W_1$ on the RHS is the Wasserstein-1 distance in $\P(\P(\mathbb R^d))$ defined in Definition 3.5 of \cite{chaos_3}. Since $\delta_{\bar\pi}$ is the Dirac measure in $\P(\P(\mathbb R^d))$, we have
\begin{equation}
	\W_1(\Law(\tilde \mu_{n\tau}),\delta_{\bar\pi}) = 
	\mathbb E\big[\W_1(\tilde \mu_{n\tau},\bar\pi)\big],~~~~
	\W_1(\Law(\tilde \mu_{n\tau}^{\RB}),\delta_{\bar\pi}) = 
	\mathbb E\big[\W_1(\tilde \mu_{n\tau}^{\RB},\bar\pi)\big].
\end{equation}

Concluding the discussion above, we have
the following equivalent form of Theorem \ref{theorem:MVP sample quality}.
\begin{corollary}
\label{corollary:MVP sample quality}
Under Assumptions \ref{assumption:bounded} and \ref{assumption:kappa}, if there exists a constant $M_4$ such that
$$
	\max_{1\Le i\Le N}\int_{\mathbb R^{Nd}} 
	|x^i|^4\nu(\d x) \Le M_4,
$$
and the constant $L_1$ in \eqref{boundedness 1} and the time step $\tau$ satisfy
$$
	L_1 < \frac{c_0\vp_0\sigma^2}{16},~~~~
	\tau<\min\bigg\{
	\frac{\alpha}{2L_0^2},\frac1{2\alpha}
	\bigg\},
$$
then there exist positive constants $\lambda = \lambda(\kappa,L_0,\sigma)$, $C_1 = C_1(\kappa,L_0,M_4,\sigma)$ and $C_2 = C_2(\kappa,L_0,\sigma)$
\begin{equation}
	\max\Big\{
	\mathbb E \big[\W_1(\tilde\mu_{n\tau},\bar\pi)\big],
	\mathbb E \big[\W_1(\tilde\mu_{n\tau}^{\RB},\bar\pi)\big]
	\Big\}
	\Le C_1\sqrt{\tau} + C_1e^{-\lambda n\tau} + \frac{C_2}{\sqrt{N}},~~~~
	\forall n\Ge0.
	\label{MVP error terms}
\end{equation}
\end{corollary}
Corollary \ref{corollary:MVP sample quality} characterizes the long-time sampling error of the numerical methods \eqref{d-IPS}\eqref{d-RB-IPS} for the MVP \eqref{MVP}. The error terms in the RHS of \eqref{MVP error terms} consist of three parts:
\begin{enumerate}
\setzero
\item $C_1\sqrt{\tau}$: time discretization and random batch divisions;
\item $C_1e^{-\lambda n\tau}$: exponential convergence of the numerical method;
\item $C_2/\sqrt{N}$: uniform-in-time propagation of chaos.
\end{enumerate}
If we aim to achieve $O(\ep)$ error in the $\W_1$ distance, then the parameters of the numerical methods should be chosen as
\begin{equation}
	N = O(\ep^{-2}),~~~~\tau = O(\ep^2),~~~~n\tau = O(\log\ep^{-1}),
\end{equation}
then the complexity of the discrete IPS \eqref{d-IPS} and the discrete RB--IPS \eqref{d-RB-IPS} is $O(\ep^{-6}\log\ep^{-1})$ and $O(\ep^{-4}\log\ep^{-1})$ respectively. In this way, the discrete RB--IPS \eqref{d-RB-IPS} consumes less complexity to achieve the desired error tolerance.

\section{Conclusion}
In this paper we have employed the triangle inequality framework to study the long-time error of the discrete RB--IPS \eqref{d-RB-IPS}, and showed that the discrete RB--IPS \eqref{d-RB-IPS} is a reliable numerical approximation to the IPS \eqref{IPS} and the MVP \eqref{MVP}. The triangle inequality framework is a flexible approach to estimate the long-time error using the geometric ergodicity and the finite-time error analysis. It is expected that such an error analysis framework can be used to estimate the long-time error of a wide class of stochastic processes.

\section*{Acknowledgements}
Z. Zhou is supported by the National Key R\&D Program of China, Project Number 2020YFA0712000, 2021YFA1001200, and NSFC grant Number 12031013, 12171013. The authors would like to thank Prof.~Zhenfu Wang and Lei Li for the helpful discussion on the propagation of chaos, and Prof.~Jian-Guo Liu for the discussion on the geometric ergodicity.

\appendix
\begin{appendices}
\section{Additional proofs for Sections 3 and 4}
\begin{proofof}\textbf{of Lemma \ref{lemma:IPS moment}}
Let us consider the IPS \eqref{IPS} first.
By Ito's formula,
\begin{equation}
	\d |X_t^i|^2 = 2X_t^i\cdot (b^i(X_t^i)+\sigma\d W_t^i) + d\sigma^2\d t.
\end{equation}
Hence
\begin{align*}
	\mathbb E|X_t^i|^2 & = \mathbb E|X_0^i|^2 + 2
	\int_0^t X_s^i\cdot b^i(X_s)\d s + d \sigma^2 t \\
	& = \mathbb E|X_0^i|^2 + 2
		\int_0^t X_s^i\cdot \big(b(X_s^i) + \gamma^i(X_s)\big)\d s + d\sigma^2 t.
\end{align*}
On the one hand, $\gamma^i$ is uniformly bounded by $L_1$, hence
\begin{equation}
	2\int_0^t X_s^i\cdot \gamma^i(X_s)\d s \Le 2L_1\int_0^t|X_s^i|\d s \Le L_1\int_0^t |X_s^i|^2\d s + L_1t.
\end{equation}
On the other hand, using the linear growth condition on $b$ one has
\begin{align}
	2\int_0^t X_s^i\cdot b(X_s^i)\d s & \Le 2L_0
	\int_0^t (|X_s^i|^2 + |X_s^i|)\d s \notag \\
	& \Le L_0 \int_0^t (3|X_s^i|^2 + 1)\d s \notag \\
	& \Le 3L_0 \int_0^t |X_s^i|^2 + L_0t.
\end{align}
Using these inequalities, one obtains
\begin{equation}
	2\int_0^t X_s^i\cdot \big(b(X_s^i) + \gamma^i(X_s)\big)\d s \Le (3L_0+L_1)
	\int_0^t |X_s^i|^2\d s + (L_0+L_1)t.
\end{equation}
Let $L:=3L_0+L_1+d\sigma^2$, then for any $t\in[0,T]$,
\begin{align*}
	\mathbb E|X_t^i|^2 & \Le \mathbb E|X_0^i|^2 + L\int_0^t|X_s^i|^2\d s + Lt \\
	& \Le M + LT + L\int_0^t|X_s^i|^2\d s.
\end{align*}
Using Gronwall's inequality,
\begin{equation}
	\mathbb E|X_t^i|^2 \Le (M+LT)\exp(LT),~~~~t\in[0,T],
\end{equation}
which yields the first inequality of \eqref{IPS moment finite time}. For the second inequality of \eqref{IPS moment finite time},
use the SDE
\begin{equation}
	X_t^i - X_{t_n}^i = \int_{t_n}^t b^i(X_s)\d s + \sigma(W_t^i - W_{t_n}^i).
\end{equation}
Hence
\begin{equation}
	\mathbb E|X_t^i - X_{t_n}^i|^2 \Le 2\mathbb E
	\bigg|
	\int_{t_n}^t b^i(X_s)\d s 
	\bigg|^2 + 2d\sigma^2\tau\Le 
	2\tau
		\int_{t_n}^t \mathbb E|b^i(X_s)|^2 \d s 
	 + 2d\sigma^2\tau
	.
\end{equation}
Using the linear growth condition
\begin{equation}
	|b^i(x)| \Le |b(x^i)| + |\gamma^i(x)| \Le L_0(|x^i|+1) + L_1 \Le L(|x^i|+1),
\end{equation}
one has $|b^i(x)|^2 \Le 2L^2(|x^i|^2+1)$ and thus
\begin{align*}
	\mathbb E|X_t^i - X_{t_n}^i|^2 & \Le 
	4L^2\tau \int_{t_n}^t \big(
	\mathbb E|X_s^i|^2+1\big)\d s + 2d\sigma^2\tau \\
	& \Le 4L^2C\tau + 4L^2\tau + 2d\sigma^2\tau = C\tau,
\end{align*}
which is exactly the desired result. The proof above also holds true for the RB--IPS \eqref{RB-IPS} because we only need to use $|\gamma^i(x)| \Le L_1$ in each time step $[t_n,t_{n+1})$.
\end{proofof}
\begin{proofof}\textbf{of Theorem \ref{theorem:IPS strong}}
WLOG assume the time step $\tau\Le T$.
Define the trajectory difference $e_n^i = X_{t_n}^i - \tilde X_n^i$. The IPS (\ref{IPS}) and the discrete IPS (\ref{d-IPS}) are given by
\begin{equation}
	X_{t_{n+1}}^i = X_{t_n}^i + \int_{t_n}^{t_{n+1}}
	b^i(X_t)\d t + \sigma W_\tau^i, ~~~~
	\tilde X_{n+1}^i = \tilde X_n + \int_{t_n}^{t_{n+1}}
	b^i(\tilde X_n)\d t + \sigma W_\tau^i,
\end{equation}
where $W_\tau^i := W_{t_{n+1}}^i - W_{t_n}^i\sim\N(0,\tau)$.
Then $e_n^i$ satisfies the recurrence relation
\begin{equation}
	e_{n+1}^i = e_n^i + \int_{t_n}^{t_{n+1}}
	(b^i(X_t)-b^i(\tilde X_n))\d t.
	\label{appendix:e_n update}
\end{equation}
Recall $b^i(x) = b(x^i) + \gamma^i(x)$.
Squaring both sides of \eqref{appendix:e_n update} we obtain
\begin{align*}
	|e_{n+1}^i|^2 & \Le (1+\tau)|e_n^i|^2 +
	\Big(1+\frac1\tau\Big)
	\bigg(
		\int_{t_n}^{t_{n+1}}
		(b^i(X_t) - b^i(\tilde X_n))\d t
	\bigg)^2 \\
	& \Le (1+\tau)|e_n^i|^2 + (1+\tau)
	\int_{t_n}^{t_{n+1}}
	|b^i(X_t) - b^i(\tilde X_n)|^2\d t \\
	& \Le (1+\tau)|e_n^i|^2 + 2(1+\tau)
	\int_{t_n}^{t_{n+1}}
		|b(X_t^i) - b(\tilde X_n^i)|^2\d t + 2(1+\tau)
		\int_{t_n}^{t_{n+1}}
		|\gamma^i(X_t) - \gamma^i(\tilde X_n)|^2\d t.
\end{align*}
On the one hand, the global Lipschitz condition of $b$ implies
\begin{equation}
	|b(X_t^i) - b(\tilde X_n^i)| \Le L_0|X_t^i - \tilde X_n^i| \Longrightarrow 
	\int_{t_n}^{t_{n+1}}
	|b(X_t^i) - b(\tilde X_n^i)|^2 \Le 
	L_0^2 \int_{t_n}^{t_{n+1}}
	|X_t^i - \tilde X_n^i|^2\d t.
	\label{appendix:IPS strong 1}
\end{equation}
On the other hand, the boundedness of $\gamma^i$ implies
\begin{align}
	|\gamma^i(X_t) - \gamma^i(\tilde X_n)| & \Le L_1|X_t^i - \tilde X_n^i| + \frac{L_1}{N-1}
	\sum_{j\neq i} |X_t^j - \tilde X_n^j| \Longrightarrow \notag \\
	|\gamma^i(X_t) - \gamma^i(\tilde X_n)|^2 & \Le 2L_1^2|X_t^i - \tilde X_n^i|^2 + 2L_1^2 \bigg(
		\frac1{N-1}\sum_{j\neq i} |X_t^j - \tilde X_n^j|
	\bigg)^2 \notag \\
	& \Le 2L_1^2|X_t^i - \tilde X_n^i|^2 + \frac{2L_1^2}{N-1}
	\sum_{j\neq i}|X_t^j-\tilde X_n^j|^2 \Longrightarrow \notag \\
	\int_{t_n}^{t_{n+1}}
	|\gamma^i(X_t) - \gamma^i(\tilde X_n)|^2\d t & \Le
	2L_1^2 \int_{t_n}^{t_{n+1}} |X_t^i - \tilde X_n^i|^2\d t + 
	\frac{2L_1^2}{N-1}\sum_{j\neq i}
	\int_{t_n}^{t_{n+1}} |X_t^j - \tilde X_n^j|^2\d t.
	\label{appendix:IPS strong 2}
\end{align}
Combining \eqref{appendix:IPS strong 1}\eqref{appendix:IPS strong 2}, $e_{n+1}^i$ has the estimate
\begin{align*}
	|e_{n+1}^i|^2 & \Le (1+\tau) |e_n^i|^2 + 
	(1+\tau) (2L_0^2+4L_1^2)\int_{t_n}^{t_{n+1}}
	|X_t^i - \tilde X_n^i|^2\d t \,+ \\
	& ~~~~(1+\tau)
	\frac{4L_1^2}{N-1}
	\sum_{j\neq i}
	\int_{t_n}^{t_{n+1}} 
	|X_t^j - \tilde X_n^j|^2\d t.
\end{align*}
Summation over $i$ gives
\begin{equation}
	\sum_{i=1}^N |e_{n+1}^i|^2 \Le 
	(1+\tau)\sum_{i=1}^N |e_n^i|^2 + (1+\tau)(2L_0^2 + 8L_1^2) \sum_{i=1}^N 
	\int_{t_n}^{t_{n+1}}|X_t^i - \tilde X_n^i|^2\d t.
	\label{appendix:IPS sum of e2}
\end{equation}
Note that
\begin{equation}
	|X_t^i - \tilde X_n^i|^2 \Le 
	2|X_t^i - X_{t_n}^i|^2 + 2|X_{t_n}^i-\tilde X_n^i|^2,
\end{equation}
from Lemma \ref{lemma:IPS moment} we have
\begin{equation}
	\mathbb E|X_t^i - \tilde X_n^i|^2\Le C\tau + 2\mathbb E|e_n^i|^2.
	\label{appendix:IPS C_tau_e}
\end{equation}
Integrating \eqref{appendix:IPS C_tau_e} in the time interval $[t_n,t_{n+1})$ gives
\begin{equation}
	\int_{t_n}^{t_{n+1}} \mathbb E
	|X_t^i - \tilde X_n^i|^2 \Le C\tau^2 + 2\tau\mathbb E|e_n^i|^2.
\end{equation}
Taking the expectation in (\ref{appendix:IPS sum of e2}) gives
\begin{align}
	\sum_{i=1}^N \mathbb E|e_{n+1}^i|^2 & \Le 
	(1+\tau)\sum_{i=1}^N \mathbb E|e_n^i|^2 + 
	C(1+\tau)\Big(
		N\tau^2 + \tau\sum_{i=1}^N\mathbb E |e_n^i|^2
	\Big) \notag \\
	& \Le (1+C\tau)\sum_{i=1}^N \mathbb E|e_n^i|^2 + CN\tau^2.
	\label{appendix:IPS strong check point}
\end{align}
Note that $e_0^i\equiv0$, the discrete Gronwall's inequality thus gives
\begin{equation}
	\frac1N\sum_{i=1}^N \mathbb E|e_n^i|^2 \Le  \tau \Big(
		(1+C\tau)^n-1
	\Big) \Le 
	e^{CT}\tau = C\tau,
\end{equation}
which implies the strong error is bounded by $C\tau$ for $0\Le n\Le T/\tau$.

Now we turn to the random batch case.
Let $e_n^i = Y_{t_n}^i - \tilde Y_n^i$, then $e_n^i$ satisfies
\begin{equation}
	|e_{n+1}^i|^2 \Le (1+\tau)|e_n^i|^2 + 2(1+\tau)
		\int_{t_n}^{t_{n+1}}
			|b(Y_t^i) - b(\tilde Y_n^i)|^2\d t + 2(1+\tau)
			\int_{t_n}^{t_{n+1}}
			|\gamma^i(Y_t) - \gamma^i(\tilde Y_n)|^2\d t.
\end{equation}
Again we stress that the perturbation force $\gamma^i(x)$ depends on the batch division $\D=\{\C_1,\cdots,\C_q\}$. Regardless of the batch division in the time interval $[t_n,t_{n+1})$, we have the inequalities
\begin{equation}
	\int_{t_n}^{t_{n+1}}
	|b(Y_t^i) - b(\tilde Y_n^i)|^2 \Le 
	L_0^2 \int_{t_n}^{t_{n+1}}
	|Y_t^i - \tilde Y_n^i|^2\d t
	\label{appendix:RB-IPS strong 1}
\end{equation}
and
\begin{equation}
	|\gamma^i(Y_t) - \gamma^i(\tilde Y_n)| \Le
		2L_1^2 \int_{t_n}^{t_{n+1}} |Y_t^i - \tilde Y_n^i|^2\d t + 
		\frac{2L_1^2}{p-1}\sum_{j\neq i,j\in\C}
		\int_{t_n}^{t_{n+1}} |Y_t^j - \tilde Y_n^j|^2\d t.
	\label{appendix:RB-IPS strong 2}
\end{equation}
Combining \eqref{appendix:RB-IPS strong 1}\eqref{appendix:RB-IPS strong 2}, $e_{n+1}^i$ has the estimate
\begin{align*}
	|e_{n+1}^i|^2 & \Le (1+\tau) |e_n^i|^2 + 
	(1+\tau) (2L_0^2+4L_1^2)\int_{t_n}^{t_{n+1}}
	|Y_t^i - \tilde Y_n^i|^2\d t \,+ \\
	& ~~~~(1+\tau)
	\frac{4L_1^2}{p-1}
	\sum_{j\neq i,j\in\C}
	\int_{t_n}^{t_{n+1}} 
	|Y_t^j - \tilde Y_n^j|^2\d t.
\end{align*}
Summation over $i\in\C$ and $\C\in\D$ recovers
\begin{equation}
	\sum_{i=1}^N |e_{n+1}^i|^2 \Le 
	(1+\tau)\sum_{i=1}^N |e_n^i|^2 + (1+\tau)(2L_0^2 + 8L_1^2) \sum_{i=1}^N 
	\int_{t_n}^{t_{n+1}}|Y_t^i - \tilde Y_n^i|^2\d t.
	\label{appendix:RB-IPS error sum}
\end{equation}
Using the same strategy with Theorem \ref{theorem:IPS strong}, we have
\begin{equation}
	\mathbb E|Y_t^i - \tilde Y_n^i|^2\Le C\tau + 2\mathbb E|e_n^i|^2.
\end{equation}
Taking the expectation in \eqref{appendix:RB-IPS error sum} the gives
\begin{align}
	\sum_{i=1}^N \mathbb E|e_{n+1}^i|^2 & \Le 
	(1+\tau)\sum_{i=1}^N \mathbb E|e_n^i|^2 + 
	C(1+\tau)\Big(
		N\tau^2 + \tau\sum_{i=1}^N\mathbb E |e_n^i|^2
	\Big) \notag \\
	& \Le (1+C\tau)\sum_{i=1}^N \mathbb E|e_n^i|^2 + CN\tau^2,
\end{align}
which is exactly the same with \eqref{appendix:IPS strong check point} in the proof of Theorem \ref{theorem:IPS strong}.
The rest part of the proof is completely the same with Theorem \ref{theorem:IPS strong}.
\end{proofof}
\begin{proofof}{\textbf{of Lemma \ref{lemma:f24}}}
(1) First we estimate $|f(x,\tau)|^2$. Using Assumptions \ref{assumption:bounded} and \ref{assumption:dissipation}, we have
\begin{align}
	|f(x,\tau)|^2 & = |x+b(x)\tau|^2 \notag \\
	& = |x|^2 + 2x\cdot b(x) \tau + |b(x)|^2\tau^2 \notag \\
	& \Le |x|^2 + 2(\theta-\alpha|x|^2)\tau + 2L_0^2(|x|^2+1)\tau^2 \\
	& = (1+2L_0^2\tau^2-2\alpha\tau)|x|^2 + 2\theta\tau + 2L_0^2\tau^2.
	\label{appendix:f21}
\end{align}
Since $\tau<\alpha/(2L_0^2)$, \eqref{appendix:f21} implies
\begin{equation}
	|f(x,\tau)|^2 \Le (1-\alpha\tau)|x|^2 + (\alpha+2\theta)\tau,
	\label{appendix:f22}
\end{equation}
To estimate $|f(x,\tau)|^4$, sqaure both sides of \eqref{appendix:f22} and utilize $\tau<1/(2\alpha)$, then
\begin{align}
	|f(x,\tau)|^4 & \Le (1-\alpha\tau)^2|x|^4 + 2(\alpha+2\theta)|x|^2\tau + (\alpha+2\theta)^2\tau^2 \notag \\
	& \Le (1-\alpha\tau)^2|x|^4 + \Big(
		k\tau|x|^4 + 
		\frac{(\alpha+2\theta)^2}{k}\tau
	\Big) + (\alpha+2\theta)^2\tau^2 \notag \\
	& \Le (1-2\alpha\tau +\alpha^2\tau^2 + k\tau)|x|^4 + 
	\frac{(\alpha+2\theta)^2}{k}\tau + (\alpha+2\theta)^2\tau^2 \notag \\
	& = \Big(1-\frac{3\alpha\tau}2 + k\tau\Big)|x|^4 + O(\tau),
	\label{appendix:f4tau}
\end{align}
where $k>0$ is an $O(1)$ parameter to be determined. Choose $k = \alpha/2$ in \eqref{appendix:f4tau}, then
\begin{equation}
	|f(x,\tau)|^4 \Le 
	(1-\alpha\tau)|x|^4 + O(\tau),
\end{equation}
hence \eqref{f24:inequality 1} holds.\\[6pt]
(2) By direct calculation,
\begin{align*}
	|f(x,\tau) + \gamma\tau|^2 & \Le 
	|f(x,\tau)|^2 + 2|f(x,\tau)|L_1\tau + L_1^2\tau^2 \\
	& \Le |f(x,\tau)|^2 + \Big(|f(x,\tau)|^2k\tau + \frac{L_1^2}k\tau\Big) + L_1^2\tau^2 \\
	& = (1+k\tau)|f(x,\tau)|^2 + O(\tau),
\end{align*}
where $k>0$ is an $O(1)$ parameter to be determined. By \eqref{f24:inequality 1}, we choose $k = \alpha/2$ and
\begin{align}
	|f(x,\tau) + \gamma\tau|^2 & \Le 
	\Big(1+\frac{\alpha\tau}2\Big)|f(x,\tau)|^2 + O(\tau) \notag \\
	& \Le \Big(1+\frac{\alpha\tau}2\Big)\Big(
		(1-\alpha\tau)|x|^2 + O(\tau)
	\Big) + O(\tau) \notag \\
	& \Le \Big(1-\frac{\alpha\tau}2\Big)|x|^2 + O(\tau).
	\label{appendix:fgamma2}
\end{align}
Squaring both sides of \eqref{appendix:fgamma2}, one obtains
\begin{align*}
	|f(x,\tau)+\gamma\tau|^4 & \Le \bigg(
		\Big(1-\frac{\alpha\tau}2\Big)|x|^2 + C\tau
	\bigg)^2 \\
	& = \bigg(1-\alpha\tau+\frac{\alpha^2\tau^2}4\bigg)|x|^4 + 2C|x|^2\tau + O(\tau^2) \\
	& \Le \Big(1-\frac{3\alpha\tau}4\Big)|x|^4 + \Big(k\tau|x|^4 + \frac{C^2\tau}{k}\Big) + O(\tau^2) \\
	& = \Big(1-\frac{3\alpha\tau}4+k\tau\Big)|x|^4 + O(\tau),
\end{align*}
where $k>0$ is a $O(1)$ parameter. By choosing $k = \alpha/4$, \eqref{f24:inequality 2} holds true.
\end{proofof}
\begin{proofof}\textbf{of Theorem \ref{theorem:moment uniform}}
The update scheme of the discrete IPS trajectory $\tilde X_n^i$ is given by
\begin{equation}
	\tilde X_{n+1}^i = \tilde X_n^i + b(\tilde X_n^i)\tau + \gamma^i(\tilde X_n)\tau + \sigma W_\tau^i,
	\label{appendix:X update}
\end{equation}
where $W_\tau^i\sim \N(0,\tau)$, and $\gamma^i$ is defined in \eqref{IPS gamma}. With $f(x,\tau) = x+b(x)\tau$, we can write \eqref{appendix:X update} as
\begin{equation}
	\tilde X_{n+1}^i = f(\tilde X_n^i,\tau) + 
	\gamma^i(\tilde X_n)\tau + \sigma W_\tau^i.
\end{equation}
Note that the random variable $W_\tau^i$ is independent of $\tilde X_n^i$, we have
\begin{align}
	\mathbb E|\tilde X_{n+1}^i|^4 & = \mathbb E|f(\tilde X_n^i,\tau) + \gamma^i(\tilde X_n)\tau|^4 + 
		6\,\mathbb E|f(\tilde X_n^i,\tau) + \gamma^i(\tilde X_n)\tau|^2\,\mathbb E|\sigma W_\tau^i|^2+
		\mathbb E|\sigma W_\tau^i|^4 \notag  \\
	& = \mathbb E|f(\tilde X_n^i,\tau) + \gamma^i(\tilde X_n)\tau|^4 + 
	6\,\mathbb E|f(\tilde X_n^i,\tau) + \gamma^i(\tilde X_n)\tau|^2d\sigma^2\tau+
	3d^2\sigma^4\tau^2 \notag \\
	& \Le \mathbb E|f(\tilde X_n^i,\tau) + \gamma^i(\tilde X_n)\tau|^4 + 
	\bigg(k\tau\,\mathbb E|f(\tilde X_n^i,\tau) + \gamma^i(\tilde X_n)\tau|^4  +\frac{9d^2\sigma^4\tau^2}{k}\bigg) + 3d^2\sigma^4\tau^2 \notag \\
	& = (1+k\tau)\mathbb E|f(\tilde X_n^i,\tau) + \gamma^i(\tilde X_n)\tau|^4 + O(\tau^2),
	\label{appendix:moment ex4}
\end{align}
where $k>0$ is an $O(1)$ parameter to be determined.
Since $\gamma^i(\tilde X_n)$ is uniformly bounded by $L_1$, by Lemma \ref{lemma:f24} we have
\begin{equation}
	\mathbb E|f(\tilde X_n^i) + \gamma^i(\tilde X_n)\tau|^4 \Le 
	\Big(1-\frac{\alpha\tau}2\Big)\mathbb E|\tilde X_n^i|^4 + C\tau.
\end{equation}
Hence \eqref{appendix:moment ex4} implies
\begin{align}
	\mathbb E|\tilde X_{n+1}^i|^4 & \Le
	(1+k\tau)\bigg(
		\Big(1-\frac{\alpha\tau}2\Big)\mathbb E|\tilde X_n^i|^4 + C\tau
	\bigg) + O(\tau^2) \notag \\
	&  \Le \bigg(
		1-\Big(\frac{\alpha}2-k\Big)\tau
	\bigg)\mathbb E|\tilde X_n^i|^4 + O(\tau).
	\label{appendix:moment ex4 2}
\end{align}
Now we can choose $k = \alpha/4$ in \eqref{appendix:moment ex4 2} to obtain
\begin{equation}
	\mathbb E|\tilde X_{n+1}^i|^4 \Le \Big(
	1-\frac{\alpha\tau}4
	\Big)\mathbb E|\tilde X_n^i|^4 + C\tau,
\end{equation}
and thus by Gronwall's inequality,
\begin{equation}
	\sup_{n\Ge0} \mathbb E|\tilde X_n^i|^4 \Le \max\bigg\{M_4,\frac{4C}\alpha\bigg\}.
\end{equation}
For the discrete RB--IPS \eqref{d-RB-IPS}, the proof is completely the same because we still have $|\gamma^i(x)|\Le L_1$ and thus the recurrence relation
\begin{equation}
	\mathbb E|\tilde Y_{n+1}^i|^4 \Le \Big(
	1-\frac{\alpha\tau}4
	\Big)\mathbb E|\tilde Y_n^i|^4 + C\tau.
\end{equation}
holds true.
\end{proofof}
\begin{proofof}\textbf{of Lemma \ref{lemma:induction}}
By induction on the integer $s\Ge1$, it is easy to verify if $n\Ge sm$, then
\begin{equation}
	a_n \Le \ep\frac{1-q^s}{1-q} + q^s a_{n-sm}.
\end{equation}
For any integer $n\Ge0$, let $n = sm + r$ for some integer $s\Ge0$ and $r\in\{0,1,\cdots,m-1\}$. Then
\begin{equation}
	a_n \Le \frac{\ep}{1-q} + Mq^s \Le \frac{\ep}{1-q} + 
	Mq^{\frac nm-1},
\end{equation}
yielding \eqref{induction an}.
\end{proofof}
\begin{proofof}\textbf{of Theorem \ref{theorem:MVP ergodicity}}
Given the probability distributions $\mu,\nu\in\P(\mathbb R^d)$, by Theorem \ref{theorem:IPS ergodicity} we have
\begin{equation}
	\W_1(\mu^{\otimes N} p_t, \nu^{\otimes N} p_t) \Le Ce^{-\beta t}\W_1(\mu^{\otimes N},\nu^{\otimes N}) = 
	Ce^{-\beta t} \W_1(\mu,\nu).
\end{equation}
Here, $(p_t)_{t\Ge0}$ is the semigroup of the IPS \eqref{IPS} in $\mathbb R^{Nd}$. Using the triangle inequality, we have
\begin{align*}
	\W_1(\mu\bar p_t,\nu\bar p_t) & = 
	\W_1(\mu^{\otimes N}\bar p_t^{\otimes N},\nu^{\otimes N}\bar p_t^{\otimes N}) \\
	& \Le \W_1(\mu^{\otimes N}\bar p_t^{\otimes N},\mu^{\otimes N} p_t) + 
	\W_1(\nu^{\otimes N}\bar p_t^{\otimes N},\nu^{\otimes N} p_t) + \W_1(\mu^{\otimes N} p_t, \nu^{\otimes N} p_t) \\
	& \Le \W_1(\mu^{\otimes N}\bar p_t^{\otimes N},\mu^{\otimes N} p_t) + \W_1(\nu^{\otimes N}\bar p_t^{\otimes N},\nu^{\otimes N} p_t) +Ce^{-\beta t} \W_1(\mu,\nu).
\end{align*}
By Theorem \ref{theorem:McKean}, for given $t>0$ there exists a constant $C_0 = C_0(\kappa,L_0,\sigma,t)$ such that
\begin{equation}
	\W_1(\mu^{\otimes N}\bar p_t^{\otimes N},\mu^{\otimes N} p_t) \Le \frac{C_0}{N}.
\end{equation}
Hence we obtain
\begin{equation}
	\W_1(\mu\bar p_t,\nu\bar p_t) \Le \frac{2C_0}{\sqrt{N}} + Ce^{-\beta t} \W_1(\mu,\nu).
\end{equation}
Fix $t>0$ and let $N\rightarrow\infty$, we obtain the desired result
\begin{equation}
	\W_1(\mu\bar p_t,\nu\bar p_t) \Le Ce^{-\beta t} \W_1(\mu,\nu).
\end{equation}
\end{proofof}
\begin{proofof}\textbf{of Corollary \ref{corollary:MVP ergodicity invariant}}
First we prove the existence of the invariant distribution $\bar\pi\in\P_1(\mathbb R^d)$ of the MVP \eqref{MVP}. Since $(\bar p_t)_{t\Ge0}$ is a nonlinear semigroup, we cannot use the same technique as in the linear case. Our proof below is partially inspired from Theorem 5.1 of \cite{invariant}.
Choose the constant $T$ which satisfies $Ce^{-\beta T} = 1/2$, then we have
\begin{equation}
	\W_1(\mu\bar p_T,\nu\bar p_T) \Le \frac12\W_1(\mu,\nu)
\end{equation}
for any probability distributions $\mu,\nu\in\P_1(\mathbb R^d)$. Hence the mapping $\mu\mapsto \mu \bar p_T$ is contractive in the complete metric space $(\P_1(\mathbb R^d),\W_1(\cdot,\cdot))$. Using the Banach fixed point theorem, there exists a unique fixed point $\bar\pi\in\P_1(\mathbb R^d)$ such that
\begin{equation}
	\bar\pi\bar p_T = \bar \pi.
\end{equation}
Since $(\bar p_t)_{t\Ge0}$ forms a semigroup, for any $t\Ge0$ we have
\begin{equation}
	\big(\bar \pi\bar p_t\big) \bar p_T= \bar \pi \bar p_t,
\end{equation}
which implies $\bar\pi\bar p_t\in\P_1(\mathbb R^d)$ is the invariant distribution of the operator $\bar p_T$. Due to the uniqueness of the invariant distribution $\bar\pi$ for the operator $\bar p_T$, we obtain
\begin{equation}
	\bar \pi \bar p_t = \bar \pi,~~~~\forall t\Ge0,
\end{equation} 
hence $\bar\pi\in\P_1(\mathbb R^d)$ is the invariant distribution of the semigroup $(\bar p_t)_{t\Ge0}$.

Next we estimate the difference between the invariant distributions $\pi,\bar\pi\in\P_1(\mathbb R^d)$.
We still choose the constant $T$ according to $Ce^{-\beta T} = 1/2$.
Using the triangle inequality, there exists a constant $C = C(\kappa,L_0,\sigma)$ such that
\begin{align*}
	\W_1(\bar\pi^{\otimes N},\pi) & = 
	\W_1(\bar\pi^{\otimes N}\bar p_{T}^{\otimes N},\pi p_{T}) \\
	& \Le 
	\W_1(\bar\pi^{\otimes N}\bar p_{T}^{\otimes N},\bar\pi^{\otimes N}p_{T}) + 
	\W_1(\bar\pi^{\otimes N}p_{T},\pi p_{T}) \\
	& \Le 
	\frac{C}{\sqrt{N}} + Ce^{-\beta T} \W_1(\bar\pi^{\otimes N},\pi) \\
	& = \frac{C}{\sqrt{N}} + 
	\frac12\W_1(\bar\pi^{\otimes N},\pi)
\end{align*}
Then $\W_1(\bar\pi^{\otimes N},\pi) \Le C/\sqrt{N}$.
\end{proofof}
\end{appendices}
\bibliography{reference.bib}

\begin{thebibliography}{10}

\bibitem{cpc_1}
Holger Fehske, Ralf Schneider, and Alexander Wei{\ss}e.
\newblock {\em Computational many-particle physics}, volume 739.
\newblock Springer, 2007.

\bibitem{cpc_2}
Daan Frenkel and Berend Smit.
\newblock {\em Understanding molecular simulation: from algorithms to
  applications}, volume~1.
\newblock Elsevier, 2001.

\bibitem{cpc_3}
Tony Lelievre and Gabriel Stoltz.
\newblock Partial differential equations and stochastic methods in molecular
  dynamics.
\newblock {\em Acta Numerica}, 25:681--880, 2016.

\bibitem{cpc_4}
Ben Leimkuhler and Charles Matthews.
\newblock Molecular dynamics.
\newblock {\em Interdisciplinary applied mathematics}, 36, 2015.

\bibitem{mean_1}
Fran{\c{c}}ois Golse.
\newblock The mean-field limit for the dynamics of large particle systems.
\newblock {\em Journ{\'e}es {\'e}quations aux d{\'e}riv{\'e}es partielles},
  pages 1--47, 2003.

\bibitem{mean_2}
Daniel Lacker.
\newblock Mean field games and interacting particle systems.
\newblock {\em Preprint}, 2018.

\bibitem{mean_3}
Charles Bordenave, David McDonald, and Alexandre Proutiere.
\newblock A particle system in interaction with a rapidly varying environment:
  Mean field limits and applications.
\newblock {\em arXiv preprint math/0701363}, 2007.

\bibitem{mean_4}
Pierre-Emmanuel Jabin and Zhenfu Wang.
\newblock Mean field limit for stochastic particle systems.
\newblock In {\em Active Particles, Volume 1}, pages 379--402. Springer, 2017.

\bibitem{MVP_1}
Khaled Bahlali, Mohamed~Amine Mezerdi, and Brahim Mezerdi.
\newblock Stability of mckean--vlasov stochastic differential equations and
  applications.
\newblock {\em Stochastics and Dynamics}, 20(01):2050007, 2020.

\bibitem{MVP_2}
Minyi Huang, Roland~P Malham{\'e}, and Peter~E Caines.
\newblock Large population stochastic dynamic games: closed-loop mckean-vlasov
  systems and the nash certainty equivalence principle.
\newblock {\em Communications in Information \& Systems}, 6(3):221--252, 2006.

\bibitem{chaos_4}
Alain-Sol Sznitman.
\newblock Topics in propagation of chaos.
\newblock In {\em Ecole d'{\'e}t{\'e} de probabilit{\'e}s de Saint-Flour
  XIX—1989}, pages 165--251. Springer, 1991.

\bibitem{chaos_3}
Louis-Pierre Chaintron and Antoine Diez.
\newblock Propagation of chaos: a review of models, methods and applications.
  i. models and methods.
\newblock 2022.

\bibitem{numerical_1}
Mireille Bossy and Denis Talay.
\newblock A stochastic particle method for the mckean-vlasov and the burgers
  equation.
\newblock {\em Mathematics of computation}, 66(217):157--192, 1997.

\bibitem{numerical_2}
Fabio Antonelli and Arturo Kohatsu-Higa.
\newblock Rate of convergence of a particle method to the solution of the
  mckean--vlasov equation.
\newblock {\em The Annals of Applied Probability}, 12(2):423--476, 2002.

\bibitem{numerical_3}
Xiaojie Ding and Huijie Qiao.
\newblock Euler--maruyama approximations for stochastic mckean--vlasov
  equations with non-lipschitz coefficients.
\newblock {\em Journal of Theoretical Probability}, 34(3):1408--1425, 2021.

\bibitem{numerical_4}
Jianhai Bao and Xing Huang.
\newblock Approximations of mckean--vlasov stochastic differential equations
  with irregular coefficients.
\newblock {\em Journal of Theoretical Probability}, pages 1--29, 2021.

\bibitem{numerical_5}
Yun Li, Xuerong Mao, Qingshuo Song, Fuke Wu, and George Yin.
\newblock Strong convergence of euler--maruyama schemes for mckean--vlasov
  stochastic differential equations under local lipschitz conditions of state
  variables.
\newblock {\em IMA Journal of Numerical Analysis}, 2022.

\bibitem{numerical_6}
Florent Malrieu.
\newblock Convergence to equilibrium for granular media equations and their
  euler schemes.
\newblock {\em The Annals of Applied Probability}, 13(2):540--560, 2003.

\bibitem{rbm_1}
Shi Jin, Lei Li, and Jian-Guo Liu.
\newblock Random batch methods (rbm) for interacting particle systems.
\newblock {\em Journal of Computational Physics}, 400:108877, 2020.

\bibitem{rbm_2}
Lei Li, Zhenli Xu, and Yue Zhao.
\newblock A random-batch monte carlo method for many-body systems with singular
  kernels.
\newblock {\em SIAM Journal on Scientific Computing}, 42(3):A1486--A1509, 2020.

\bibitem{rbm_3}
Shi Jin, Lei Li, Zhenli Xu, and Yue Zhao.
\newblock A random batch ewald method for particle systems with coulomb
  interactions.
\newblock {\em SIAM Journal on Scientific Computing}, 43(4):B937--B960, 2021.

\bibitem{rbm_4}
Xuda Ye and Zhennan Zhou.
\newblock Efficient sampling of thermal averages of interacting quantum
  particle systems with random batches.
\newblock {\em The Journal of Chemical Physics}, 154(20):204106, 2021.

\bibitem{rbm_5}
Fran{\c{c}}ois Golse, Shi Jin, and Thierry Paul.
\newblock The random batch method for $ n $-body quantum dynamics.
\newblock {\em arXiv preprint arXiv:1912.07424}, 2019.

\bibitem{pem_1}
Lei Li, Jian-Guo Liu, and Yijia Tang.
\newblock Some random batch particle methods for the poisson-nernst-planck and
  poisson-boltzmann equations.
\newblock {\em arXiv preprint arXiv:2004.05614}, 2020.

\bibitem{pem_2}
Jos{\'e}~Antonio Carrillo, Shi Jin, and Yijia Tang.
\newblock Random batch particle methods for the homogeneous landau equation.
\newblock {\em arXiv preprint arXiv:2110.06430}, 2021.

\bibitem{pem_3}
Seung-Yeal Ha, Shi Jin, Doheon Kim, and Dongnam Ko.
\newblock Convergence toward equilibrium of the first-order consensus model
  with random batch interactions.
\newblock {\em Journal of Differential Equations}, 302:585--616, 2021.

\bibitem{rbme_1}
Shi Jin, Lei Li, and Jian-Guo Liu.
\newblock Convergence of the random batch method for interacting particles with
  disparate species and weights.
\newblock {\em SIAM Journal on Numerical Analysis}, 59(2):746--768, 2021.

\bibitem{rbme_2}
Shi Jin, Lei Li, Xuda Ye, and Zhennan Zhou.
\newblock Ergodicity and long-time behavior of the random batch method for
  interacting particle systems.
\newblock {\em arXiv preprint arXiv:2202.04952}, 2022.

\bibitem{reflect_1}
Andreas Eberle.
\newblock Reflection coupling and wasserstein contractivity without convexity.
\newblock {\em Comptes Rendus Mathematique}, 349(19-20):1101--1104, 2011.

\bibitem{reflect_2}
Andreas Eberle.
\newblock Reflection couplings and contraction rates for diffusions.
\newblock {\em Probability theory and related fields}, 166(3):851--886, 2016.

\bibitem{tri_2}
Jonathan~C Mattingly, Andrew~M Stuart, and Desmond~J Higham.
\newblock Ergodicity for sdes and approximations: locally lipschitz vector
  fields and degenerate noise.
\newblock {\em Stochastic processes and their applications}, 101(2):185--232,
  2002.

\bibitem{tri_3}
Jonathan~C Mattingly, Andrew~M Stuart, and Michael~V Tretyakov.
\newblock Convergence of numerical time-averaging and stationary measures via
  poisson equations.
\newblock {\em SIAM Journal on Numerical Analysis}, 48(2):552--577, 2010.

\bibitem{tri_4}
Alain Durmus and Eric Moulines.
\newblock Nonasymptotic convergence analysis for the unadjusted langevin
  algorithm.
\newblock {\em The Annals of Applied Probability}, 27(3):1551--1587, 2017.

\bibitem{chaos_1}
Maxime Hauray and St{\'e}phane Mischler.
\newblock On kac's chaos and related problems.
\newblock {\em Journal of Functional Analysis}, 266(10):6055--6157, 2014.

\bibitem{asp}
E~Weinan, Tiejun Li, and Eric Vanden-Eijnden.
\newblock {\em Applied stochastic analysis}, volume 199.
\newblock American Mathematical Soc., 2021.

\bibitem{var_1}
Martin Hairer and Jonathan~C Mattingly.
\newblock Yet another look at harris’ ergodic theorem for markov chains.
\newblock In {\em Seminar on Stochastic Analysis, Random Fields and
  Applications VI}, pages 109--117. Springer, 2011.

\bibitem{reflect_3}
Andreas Eberle, Arnaud Guillin, and Raphael Zimmer.
\newblock Quantitative harris-type theorems for diffusions and mckean--vlasov
  processes.
\newblock {\em Transactions of the American Mathematical Society},
  371(10):7135--7173, 2019.

\bibitem{invariant}
Jos{\'e}~A Ca{\~n}izo and St{\'e}phane Mischler.
\newblock Harris-type results on geometric and subgeometric convergence to
  equilibrium for stochastic semigroups.
\newblock {\em arXiv preprint arXiv:2110.09650}, 2021.

\bibitem{tri_1}
Tony Shardlow and Andrew~M Stuart.
\newblock A perturbation theory for ergodic markov chains and application to
  numerical approximations.
\newblock {\em SIAM journal on numerical analysis}, 37(4):1120--1137, 2000.

\bibitem{Talagrand}
Sergey~G Bobkov, Ivan Gentil, and Michel Ledoux.
\newblock Hypercontractivity of hamilton--jacobi equations.
\newblock {\em Journal de Math{\'e}matiques Pures et Appliqu{\'e}es},
  80(7):669--696, 2001.

\bibitem{reflect_HMC}
Nawaf Bou-Rabee, Andreas Eberle, and Raphael Zimmer.
\newblock Coupling and convergence for hamiltonian monte carlo.
\newblock {\em The Annals of applied probability}, 30(3):1209--1250, 2020.

\bibitem{mvpp_1}
Gon{\c{c}}alo dos Reis, Stefan Engelhardt, and Greig Smith.
\newblock Simulation of mckean--vlasov sdes with super-linear growth.
\newblock {\em IMA Journal of Numerical Analysis}, 42(1):874--922, 2022.

\bibitem{uniform_log_Sobolev}
Arnaud Guillin, Wei Liu, Liming Wu, and Chaoen Zhang.
\newblock Uniform poincar{\'e} and logarithmic sobolev inequalities for mean
  field particle systems.
\newblock {\em The Annals of Applied Probability}, 32(3):1590--1614, 2022.

\bibitem{uniform_chaos}
Alain Durmus, Andreas Eberle, Arnaud Guillin, and Raphael Zimmer.
\newblock An elementary approach to uniform in time propagation of chaos.
\newblock {\em Proceedings of the American Mathematical Society},
  148(12):5387--5398, 2020.

\bibitem{log_Sobolev}
Arnaud Guillin, Wei Liu, Liming Wu, and Chaoen Zhang.
\newblock Uniform poincar $\{$$\backslash$'e$\}$ and logarithmic sobolev
  inequalities for mean field particles systems.
\newblock {\em arXiv preprint arXiv:1909.07051}, 2019.

\bibitem{small_diffusion}
Pierre Del~Moral and Julian Tugaut.
\newblock Uniform propagation of chaos for a class of inhomogeneous diffusions.
\newblock Technical report, Citeseer, 2014.

\bibitem{var_2}
Alain Durmus and Eric Moulines.
\newblock Nonasymptotic convergence analysis for the unadjusted langevin
  algorithm.
\newblock {\em The Annals of Applied Probability}, 27(3):1551--1587, 2017.

\bibitem{mixing}
SA~Klokov and A~Yu Veretennikov.
\newblock On mixing and convergence rates for a family of markov processes
  approximating sdes.
\newblock 2006.

\bibitem{ddsde}
Feng-Yu Wang.
\newblock Distribution dependent sdes for landau type equations.
\newblock {\em Stochastic Processes and their Applications}, 128(2):595--621,
  2018.

\bibitem{McKean}
Henry~P McKean.
\newblock Propagation of chaos for a class of non-linear parabolic equations.
\newblock {\em Stochastic Differential Equations (Lecture Series in
  Differential Equations, Session 7, Catholic Univ., 1967)}, pages 41--57,
  1967.

\bibitem{rbme_3}
Shi Jin and Lei Li.
\newblock On the mean field limit of the random batch method for interacting
  particle systems.
\newblock {\em Science China Mathematics}, 65(1):169--202, 2022.

\bibitem{elementary}
Alain Durmus, Andreas Eberle, Arnaud Guillin, and Raphael Zimmer.
\newblock An elementary approach to uniform in time propagation of chaos.
\newblock {\em Proceedings of the American Mathematical Society},
  148(12):5387--5398, 2020.

\end{thebibliography}
\bibliographystyle{unsrt}
\end{document}